\documentclass[12pt]{article}
\usepackage{amsmath, amsthm, amssymb}
\usepackage{hyperref}
\usepackage{verbatim}
\usepackage[top=1.0in, bottom=1.0in, left=1.0in, right=1.0in]{geometry}

\def\to{\rightarrow}

\usepackage{tkz-graph}
\usetikzlibrary{arrows}
\usetikzlibrary{shapes}
\usepackage[position=bottom]{subfig}

\usepackage{longtable}
\usepackage{array}

\usepackage{sectsty}
\allsectionsfont{\sffamily}

\makeatletter
\newtheorem*{rep@theorem}{\rep@title}
\newcommand{\newreptheorem}[2]{
\newenvironment{rep#1}[1]{
 \def\rep@title{#2 \ref{##1}}
 \begin{rep@theorem}}
 {\end{rep@theorem}}}
\makeatother

\theoremstyle{plain}
\newtheorem{thm}{Theorem}[section]
\newreptheorem{thm}{Theorem}

\newreptheorem{prop}{Proposition}
\newtheorem{lem}[thm]{Lemma}
\newreptheorem{lem}{Lemma}

\newreptheorem{conjecture}{Conjecture}
\newtheorem{conj}[thm]{Conjecture}
\newreptheorem{con}{Conjecture}
\newtheorem{cor}[thm]{Corollary}
\newreptheorem{cor}{Corollary}

\newtheorem*{LineThm}{Theorem 3.1}

\theoremstyle{definition}
\newtheorem{defn}{Definition}
\theoremstyle{remark}

\newcommand{\IN}{\mathbb{N}}

\newcommand{\chil}{\chi_{\ell}}
\newcommand{\chip}{\chi_{OL}}
\newcommand{\ch}{{\textrm{ch}}}
\newcommand{\mad}{{\textrm{mad}}}

\newcommand{\claim}[2]{{\bf Claim #1.}~{\it #2}~~}

\newcommand{\subclaim}[2]{{\bf Subclaim #1.}~{\it #2}~~}

\newcommand{\set}[1]{\left\{ #1 \right\}}

\newcommand{\setbs}[2]{\left\{ #1 \mid #2 \right\}}
\newcommand{\card}[1]{\left|#1\right|}

\newcommand{\ceil}[1]{\left\lceil#1\right\rceil}

\newcommand{\irange}[1]{\left[#1\right]}
\newcommand{\join}[2]{#1 \mbox{\hspace{2 pt}$\vee$\hspace{2 pt}} #2}

\newcommand{\parens}[1]{\left( #1 \right)}

\newcommand{\DefinedAs}{\mathrel{\mathop:}=}

\def\nats{\mathbb{Z}^+}
\def\adj{\leftrightarrow}
\def\nonadj{\not\!\leftrightarrow}
\def\ad{\textrm{ad}}
\def\mad{\textrm{mad}}

\begin{document}
\title{List-coloring claw-free graphs with $\Delta$-1 colors}
\author{Daniel W. Cranston\thanks{Department of Mathematics and Applied
Mathematics, Viriginia Commonwealth University, Richmond, VA;
\texttt{dcranston@vcu.edu}; 
Research of the first author is partially supported by NSA Grant
98230-15-1-0013.}
\and
Landon Rabern\thanks{LBD Data Solutions, Lancaster, PA;
\texttt{landon.rabern@gmail.com}}
}
\maketitle

\begin{abstract}
We prove that if $G$ is a quasi-line graph with $\Delta(G)>\omega(G)$ and
$\Delta(G)\ge 69$, then $\chip(G)\le \Delta(G)-1$.  Together with our previous
work, this implies that if $G$ is a claw-free graph with $\Delta(G)>\omega(G)$
and $\Delta(G)\ge 69$, then $\chil(G)\le \Delta(G)-1$.
\end{abstract}

%\begin{comment}
\section{Introduction}
Brooks' Theorem shows that to color a graph $G$ with $\Delta(G)$ colors, the
obvious necessary condition (no clique of size $\Delta(G)+1$) is also
sufficient, when $\Delta(G)\ge 3$.  Borodin and Kostochka~\cite{BK77}
conjectured something similar for $\Delta(G)-1$ colors.

\begin{conj}[Borodin--Kostochka~\cite{BK77}]
\label{BK-conj}
If $G$ is a graph with $\Delta(G)\ge 9$ and no clique of size $\Delta(G)$,
then $\chi(G)\le \Delta(G)-1$.
\end{conj}

This conjecture is a major open problem and has been the subject of much research.
Reed~\cite{Reed99} used probabilistic techniques to prove the conjecture when
$\Delta(G)\ge 10^{14}$.  For graphs with smaller maximum degree, the best result
\cite{big-cliques} is that $\chi(G)\le\Delta(G)-1$ whenever $G$ has no clique of
size $\Delta(G)-3$.  We have also proved Conjecture~\ref{BK-conj} for claw-free
graphs~\cite{BK-claw-free}.

Although the Borodin--Kostochka conjecture is far from resolved, it is natural
to pose the analogous
conjectures for list-coloring and online list-coloring, %(both defined below),
replacing $\chi(G)$ in Conjecture~\ref{BK-conj} with $\chil(G)$ and $\chip(G)$.
These conjectures first appeared in print in~\cite{BK-claw-free}
and~\cite{BK-squares}, respectively.  In the case of Brooks' Theorem, the
analogues for $\chil(G)$ and $\chip(G)$ both hold.  In fact, the proof of the
latter~\cite{Brooks-AT} constructs an orientation of $E(G)$ from which the
result follows by the Alon--Tarsi Theorem.
The present paper applies this approach to the Borodin--Kostochka conjecture.
More precisely, given a graph with $\Delta(G)\ge 9$ and $\omega(G)<\Delta(G)$,
we seek an orientation of $E(G)$ that implies that $\chip(G)\le \Delta(G)-1$.
Our main result is the following.

\begin{thm}
\label{MainThm}
Let $G$ be a quasi-line graph with $\Delta(G)\ge 69$.  If $\omega(G) <
\Delta(G)$, then $\chip(G)\le \Delta(G)-1$.  Further, Painter has a natural
winning strategy, using a combination of the Alon--Tarsi Theorem and the
kernel method.
\end{thm}
\bigskip

%Our goal in this paper is to prove that $\AT(G)\le \Delta(G)-1$ or $\KP(G)\le
%\Delta(G)-1$.
Chudnovsky and Seymour~\cite{CS-survey,CS-claw-free4} proved a structure theorem
for claw-free graphs.  Although it is rather complicated, it implies the
following structure theorem for quasi-line graphs, which is much simpler.
(We define the undefined terms in Section~\ref{defns}.)

\begin{thm}[\cite{CS-survey,CS-claw-free4}]
\label{QuasilineStructure}
Every connected quasi-line graph not containing a non-linear homogeneous pair
of cliques is a circular interval graph or a composition of linear interval strips.
\end{thm}

Theorem~\ref{QuasilineStructure} suggests a natural approach to prove
Theorem~\ref{MainThm}.
Let $G$ be a quasi-line graph with $\Delta(G)\ge 69$ and $\omega(G)<\Delta(G)$.
In Sections~\ref{circular-interval-graphs} and~\ref{homogeneous-pairs} we show
that if $G$ is a minimal counterexample to our theorem, then $G$ cannot be a
circular interval graph and $G$ cannot contain a non-linear homogeneous pair of
cliques.  In Section~\ref{2joins}, we consider compositions of linear
interval strips (which include line graphs, as a special case).  
We show that
a minimal counterexample $G$ must be formed from a line graph $G'$ by deleting
some (possibly empty) matching $M$.  Further, we can choose $G'$ such that
$\Delta(G')=\Delta(G)$ and $\omega(G')<\Delta(G)$.  So we prove the desired
result for all quasi-line graphs if we prove it for line graphs.  Finally, in
Section~\ref{line-graphs} we prove the theorem for line graphs.  
By combining this result with Theorem~5.6 from~\cite{BK-claw-free}, we get that
every claw-free graph $G$ with $\Delta(G)\ge 69$ and $\omega(G)<\Delta(G)$
satisfies $\chil(G)\le \Delta(G)-1$.
In other words, for these graphs we prove the
list-coloring version of the Borodin--Kostochka conjecture.

It is somewhat surprising that line graphs seem to be the most difficult case
in the proof.  In our reduction from general quasi-line graphs to line graphs,
we only need $\Delta(G)\ge 9$.  It is our proof of Theorem~\ref{MainThm} for
line graphs that requires $\Delta(G)\ge 69$.
As noted above, Theorem~5.6\ in~\cite{BK-claw-free} shows that if
$\chil(G)\le\Delta(G)-1$ for all quasi-line graphs $G$ with
$\omega(G)<\Delta(G)$ and $\Delta(G)\ge 9$, then the same bound holds for all
such claw-free graphs.
In unpublished work, we have extended this reduction to \emph{online} list-coloring.  
Thus, $\chip(G)\le\Delta(G)-1$ for every such claw-free graph $G$
with $\Delta(G)\ge 69$.  (And again, the hypothesis $\Delta(G)\ge 69$ is needed
only for the case of line graphs.)

\subsection{Definitions}
\label{defns}
Most of our terminology and notation are standard.
We write $N(v)$ for the neighborhood of a vertex $v$, and $N[v]=N(v)\cup\{v\}$.
%The neighborhood of a vertex $v$, denoted $N(v)$ is its set of adjacent
%vertices.  
When $u$ and $v$ are adjacent, we write $u\adj v$; otherwise, $u\nonadj v$.
We write $[t]$ for $\{1,\ldots,t\}$ (but we reserve, for example, $[13]$ for
citations).  The degree, $d(v)$, is the size of $N(v)$ and $d_H(v)$ is the size
of $N(v)\cap V(H)$, for any subgraph $H$.  A graph is \emph{$k$-degenerate} if
every subgraph $H$ contains a vertex $v$ with $d_H(v)\le k$.
The complement of $G$ is denoted
$\overline{G}$.  The maximum degree and clique number of $G$
are denoted $\Delta(G)$ and $\omega(G)$, and we may write $\Delta$ and $\omega$
when the context is clear.  
The chromatic number of $G$ is $\chi(G)$.  Similarly, the list chromatic 
and online list-chromatic numbers are $\chil(G)$ and $\chip(G)$.
The edge chromatic number of $G$ is $\chi'(G)$, and $\chil'(G)$ and $\chip'(G)$
are defined analogously.  A graph $G$ is \emph{$L$-colorable} if $G$ has a proper
coloring $\varphi$ such that $\varphi(v)\in L(v)$ for all $v\in V(G)$.
A graph $G$ is \emph{$f$-choosable} if $G$ is $L$-colorable whenever $|L(v)|\ge
f(v)$ for all $v$, and $f$-paintable is defined analogously.  We write $d_1$
for the function $f(v)=d(v)-1$, and thus define \emph{$d_1$-choosable} and
\emph{$d_1$-paintable}.

The subgraph of a graph $G$ induced by vertex set $S$ is $G[S]$.  
The \emph{average degree}, $\ad(G)$, of a graph $G$ is $2|E(G)|/|V(G)|$.  The
\emph{maximum average degree}, $\mad(G)$, is the maximum of $\ad(H)$, taken over
all subgraphs $H$ of $G$.
A graph or subgraph is \emph{complete} if it
induces a clique; otherwise it is \emph{incomplete}.  A graph is \emph{almost
complete} if deleting some vertex yields a complete graph.  The \emph{join}
of graphs $G$ and $H$, denoted $\join{G}{H}$, is formed from their disjoint union
by adding every edge with one endpoint in each of $G$ and $H$.

A \emph{linear interval graph} is one in which the vertices can be 
placed on the real line so that for each vertex $v$ its neighborhood is
precisely the vertices in some interval of the line containing $v$.  A
\emph{circular interval graph} is defined analogously, except that now the
vertices are placed on the unit circle.  The \emph{line graph} $G$ of some graph
$H$ has $V(G)=E(H)$ and $uv\in E(G)$ whenever $u,v\in V(G)$ and they correspond
to edges in $H$ sharing an endpoint.  A graph $G$ is \emph{quasi-line} if for
each vertex $v\in V(G)$, the subgraph $G[N(v)]$ can be covered by two cliques.
A graph if \emph{claw-free} if it contains no induced copy of $K_{1,3}$.  It is
easy to check that the class of claw-free graphs properly contains the class of
quasi-line graphs, which in turn properly contains the class of line graphs.

A \emph{homogeneous pair of cliques} $(A_1, A_2)$ in a graph $G$ is a pair of
disjoint nonempty cliques such that for each $i \in \{1,2\}$, every vertex in
$G \setminus (A_1 \cup A_2)$ is either adjacent to all of $A_i$ or non-adjacent
to all of $A_i$ and $\card{A_1} + \card{A_2} \geq 3$. The pair $(A_1, A_2)$ is
\emph{non-linear} if $G[A_1\cup A_2]$ contains an induced 4-cycle.

Chudnovsky and Seymour \cite{CS-survey} generalized the class of line graphs by
introducing the notion of \emph{compositions of strips} (\cite[Chapter~5]{king2009claw} gives a more detailed introduction).
We use the modified definition from King and
Reed \cite{king2008bounding}. A \emph{strip} $(H, A_1, A_2)$ is a claw-free
graph $H$ containing two cliques $A_1$ and $A_2$ such that for each $i \in
\{1,2\}$ and $v \in A_i$, the set $N_H(v)\setminus A_i$ is a clique.  
If $H$ is a linear interval graph with $A_1$ and $A_2$ on opposite ends, then
$(H, A_1, A_2)$ is a $\emph{linear interval strip}$.  Now let $H$ be a directed
multigraph (possibly with loops) and suppose for each edge $e$ of $H$ we have a
strip $(H_e, X_e, Y_e)$.  For each $v \in V(H)$ define

\[
C_v \DefinedAs \parens{\bigcup \setbs{X_e}{\text{$e$ is directed out of $v$}}}
\cup \parens{\bigcup \setbs{Y_e}{\text{$e$ is directed into $v$}}}.
\]

The graph formed by taking the disjoint union of $\setbs{H_e}{e \in E(H)}$ and
making $C_v$ a clique for each $v \in V(H)$ is the composition of the strips
$(H_e, X_e, Y_e)$.  Any graph formed in this way is a \emph{composition of
strips}.  Notice that if each strip $(H_e, X_e, Y_e)$ in the composition has
$V(H_e) = X_e = Y_e$, then the graph formed is just the line graph of the
multigraph formed by replacing each $e \in E(H)$ with $\card{H_e}$ copies of
$e$.

It is convenient to have notation and terminology for a strip together with
how it attaches to the graph. An \emph{interval $2$-join} in a graph $G$ is an
induced subgraph $H$ such that (i) $H$ is a nonempty linear interval graph,
(ii) the ends of $H$ are cliques $A_1$ and $A_2$, not necessarily disjoint,
(iii) $G\setminus H$ contains cliques $B_1$, $B_2$ (not necessarily disjoint) such that
$A_1$ is joined to $B_1$ and $A_2$ is joined to $B_2$, and
(iv) no other edges exist between $H$ and $G-H$.
Since $A_1, A_2, B_1, B_2$ are uniquely determined by $H$, we can 
refer to the interval $2$-join as either $H$ or, equivalently, as the quintuple
$(H, A_1, A_2, B_1, B_2)$. 

An interval $2$-join $(H, A_1, A_2, B_1, B_2)$ is \emph{trivial} if $V(H) = A_1
= A_2$ and \emph{canonical} if $A_1 \cap A_2 =
\emptyset$.  A canonical interval $2$-join $(H, A_1, A_2, B_1, B_2)$ with
leftmost vertex $v_1$ and rightmost vertex $v_t$ is \emph{reducible} if $H$ is
incomplete and $N_H(A_1)\setminus A_1 = N_H(v_1)\setminus A_1$ or
$N_H(A_2)\setminus A_2 = N_H(v_t)\setminus A_2$.  We call such a canonical
interval $2$-join reducible because we can \emph{reduce} it as follows.  
Suppose $H$ is incomplete and $N_H(A_1)\setminus A_1 = N_H(v_1)\setminus A_1$. Let $C
\DefinedAs N_H(v_1) \setminus A_1$, let $A_1' \DefinedAs C \setminus A_2$,
and let $A_2' \DefinedAs A_2 \setminus C$.  Since $H$ is incomplete, $v_t \in
A_2'$, so $H' \DefinedAs G[A_1' \cup A_2']$ is a nonempty linear interval
graph that gives the reduced canonical interval $2$-join $(H', A_1', A_2', A_1
\cup \parens{C \cap A_2}, B_2 \cup \parens{C \cap A_2})$.  

Note that reducing an
interval 2-join yields an interval 2-join with a smaller strip.  Note also that
reducing a canonical interval 2-join again yields a canonical interval 2-join.
The process of reducing a 2-join
allows us to refine the composition representation, and to get a representation
with more strips.  In particular, in a representation with the maximum number of
strips, every 2-join is irreducible.

%\RED{
%We should define the following: 
%claw-free, quasi-line, homogeneous pair,
%nonlinear homogeneous pair,
%circular interval graph, linear interval strip, 
%graph polynomial,
%$(f_{k_1,\ldots,k_n})$, Combinatorial Nullstellensatz. 
%join, $\overline{G}$,
%$\mad(G)$, $\chi'(G)$, $\chil'(G)$, $\chip'(G)$,  
%$k$-degenerate.
%Kernel, kernel-perfect.
%
%}

\subsection{Coloring from Graph Orientations}
In this section, we show how we can orient the edges of a graph $G$ to prove
upper bounds on $\chil(G)$ and $\chip(G)$.  It is well-known that if a graph
$G$ is $k$-degenerate, then $\chi(G)\le k+1$; and this upper bound holds also
for $\chil(G)$ and $\chip(G)$.  If $v_1,\ldots,v_n$ is a vertex order such that
each $v_i$ has at most $k$ neighbors with smaller index, then we use at most
$k+1$ colors when we color greedily in order of increasing index.
We can view this bound in terms of orientations as follows.  Orient each edge
$v_iv_j$ as $v_i\to v_j$ when $i>j$.  Now $\chi(G)\le k+1$ whenever $G$ has an
acyclic orientation $D$ with maximum outdegree $k$.
Alon and Tarsi proved the following far-reaching generalization, where $D$ need
not be acyclic.

\begin{thm}[Alon--Tarsi~\cite{AlonTarsi}]
\label{AT-thm}
Let $f:V\to \nats$ be a list size assigment, and let $D$ be an orientation of
$E(G)$ in which $|EE(D)|\ne|EO(D)|$,  where $EE(D)$ and $EO(D)$ are the sets of
spanning Eulerian subgraphs of $D$ with an even (resp.~odd) number of edges.
If $f(v) > d_D^+(v)$ for all $v\in V(G)$, then $G$ is $f$-choosable.  (In fact,
$f$-paintable.)
\end{thm}

Now we consider the other standard technique for coloring graphs via
orientations.
A \emph{kernel} of a digraph $D$ is an independent set $I$ such that each vertex
not in $I$ has an out-neighbor in $I$.  A digraph is \emph{kernel-perfect} if
every induced subgraph has a kernel.  Most applications of kernels to
list-coloring use the following lemma of Bondy, Boppana, and Siegel.

\begin{lem}[Bondy--Boppana--Siegel~{\cite[Remark 2.4]{AlonTarsi}},{\cite[Lemma
2.1]{Galvin}}]
\label{BBS-lemma}
Let $f:V\to \nats$ be a list size assigment, and let $D$ be a kernel-perfect 
orientation of $E(G)$.  If $f(v) > d_D^+(v)$ for all $v\in V(G)$, then $G$ is
$f$-choosable.  (In fact, $f$-paintable.)
\end{lem}

We can easily prove Lemma~\ref{BBS-lemma} by induction.  Given such an
orientation, on each round Painter
chooses as his independent set a kernel of the subgraph induced by the vertices
listed by Lister.  This technique is called the Kernel Method.
Both Theorem~\ref{AT-thm} and Lemma~\ref{BBS-lemma} were originally proved for
list coloring, and then extended to online list-coloring by
Schauz~\cite{Schauz-Paint-Correct}.  (The extension of Lemma~\ref{BBS-lemma} has
the same proof as the original.  However, the extension of Theorem~\ref{AT-thm}
requires significant work.)

Our proofs in this paper rely heavily on both of these techniques, so the following
definitions are useful.  A graph $H$ is \emph{$f$-AT} if it has an orientation
$D$ with $f(v) > d_D^+(v)$ for all $v \in V(H)$ and with different numbers of
even and odd spanning Eulerian subgraphs.  Such a $D$ is an \emph{Alon--Tarsi
orientation} for $f$ and $H$.  A graph $H$ is
$f$-KP if some supergraph $H'$ of $H$ has a kernel-perfect orientation
where $f(v) > d^+(v)$ for all $v \in V(H')$.  Allowing this supergraph for KP
gives us more power.  For example, $K_4-e$ has no kernel-perfect orientation
showing it is degree-choosable, but if we double the edge in two triangles,
then there is such an orientation. We could allow a supergraph for AT as well, but
this doesn't give us any more power, as we will see in Lemma \ref{subgraphLemma}.  
Since our focus in this paper is the Borodin--Kostochka conjecture, we have one
more definition.

%\textcolor{red}{
%The following is probably false, but we have no counterexample and it would be
%very nice if true.  If not true, are there other general conditions we can
%place on orientations that together with the out-degree condition imply
%$f$-paintability?
%\begin{conjecture}
%If $G$ is $f$-paintable, then $G$ is $f$-AT or $f$-KP.
%\end{conjecture}
%}

\begin{defn}
A connected graph $G$ is \emph{BK-free} if it does not contain an induced
subgraph $H$ that is $f_H$-AT or $f_H$-KP where $f_H(v) \DefinedAs d_H(v) - 1 +
\Delta(G) - d_G(v)$ for all $v \in V(H)$.
\end{defn}

The motivation for this definition is that any minimal counterexample to
Theorem~\ref{MainThm} must be
BK-free.  To see this for list-coloring is easy.  Suppose $G$ is not BK-free;
say it contains subgraph $H$ that is $f$-AT or $f$-KP.  By minimality, color
$G\setminus H$.  Now, by definition, we can extend the coloring to $H$.
The same idea works for online list coloring.  On each round, Painter first
plays optimally on $G\setminus H$, then plays optimally on $H$ (omitting from
$H$ any vertices with neighbors in $G$ that Painter chose on that round).
So in particular, if $G$ is BK-free, then $\delta(G) \ge \Delta(G) - 1$.
Thus, a vertex $v$ is \emph{high} if $d(v)=\Delta(G)$ and \emph{low} if
$d(v)=\Delta(G)-1$.
When we write that subgraph $H$ is $f$-AT or $f$-KP without specifying $f$,
we mean $f(v)=d_H(v)-1+\Delta(G)-d_G(v)$ (so $f(v)=d_H(v)-1$ when $v$ is high
in $G$ and $f(v)=d_H(v)$ when $v$ is low).
\bigskip

A special case of the weak perfect graph theorem states that if $G$ is the
complement of a bipartite graph, then $\chi(G)=\omega(G)$.
In this section, we prove a strengthening of the analogous statement for
Alon--Tarsi orientations.
This result plays a key role in Section~\ref{homogeneous-pairs}, where we handle
non-linear homogeneous pairs of cliques.

It is well known that for a graph $G$, if $H\subseteq G$, then $\chi(H)\le
\chi(G)$ and $\chil(H)\le\chil(G)$.  More generally, if $f$ is a list-size
assignment and $G$ is $f$-choosable or $f$-paintable, then so is $H$.  It is
natural to expect that an analogous statement holds for Alon--Tarsi
orientations.  Indeed it does, as we show in Lemma~\ref{subgraphLemma}.

Given a graph $G$, let $v_1,\ldots,v_n$ be an arbitrary ordering of $V(G)$.
The \emph{graph polynomial}, $g$, of $G$ is given by $g=\prod_{v_iv_j\in
E(G), i<j}(x_i-x_j)$.  Note that $g$ is independent of the ordering of $V(G)$, up to a
factor of $\pm1$.  
For a polynomial $g\in\mathbb{F}[x_1,\ldots,x_n]$, we write
$g_{k_1,\ldots,k_n}$ for the coefficient in $g$ of $x_1^{k_1}\cdots x_n^{k_n}$.
Alon and Tarsi~\cite[Corollary 2.3]{AlonTarsi} observed that $G$ is $f$-AT 
precisely when there exist $k_i$ such that $f(v_i)\ge k_i+1$ for all $i$ and
the graph polynomial has $g_{k_1,\ldots,k_n}\ne 0$.

\begin{lem}
If a graph $G$ is $f$-AT (for any particular function $f$) and $e \in E(G)$,
then $G-e$ is also $f$-AT.  More generally, if $H$ is a subgraph of $G$ and $G$
is $f$-AT, then so is $H$.
\label{subgraphLemma}
\end{lem}
\begin{proof}
%Maybe there is a good way to see this in the eulerian subgraph
%formulation, but i didn't see it.  Instead, if we back up to what all
%this means about polynomials, i think it works like this:

The second statement follows from the first by induction on $|E(G)\setminus
E(H)|$.  If $|V(H)|<|V(G)|$, then for each vertex in $V(G)\setminus V(H)$, we
first delete all of its incident edges.  Now adding or removing an isolated
vertex $v$ has no effect on the graph polynomial.

Fix a graph $G$ and a function $f$ such that $G$ is $f$-AT.
As noted above, $G$ is $f$-AT if and only if the graph polynomial of $G$ has a
nonzero term $x_1^{k_1}\cdots x_n^{k_n}$, where $x_i$ is the variable corresponding to
vertex $v_i$, such that $f(v_i) > k_i$ for all $i$.

Suppose the lemma is false, that is, there exists $e\in E(G)$ such that $G-e$
is not $f$-AT.  Let $v_1$ and $v_2$ denote the endpoints of $e$.
Since $G-e$ is not $f$-AT, its graph polynomial has no nonzero term as above. 
That is, for every term  $x_1^{j_1} \cdots x_n^{j_n}$, there exists $i$ such that 
$f(v_i) \le j_i$.  Now the graph polynomial of $G$ is formed from that of $G-e$
by multiplying by $(x_1 - x_2)$.  Terms may cancel, but the exponents never go
down.  Since $G-e$ is not $f$-AT, for every term $x_1^{j_1} \cdots x_n^{j_n}$ in
the polynomial of $G-e$, there exists $i$ such that $f(v_i) \le j_i$. 
Thus, for any remaining term $x_1^{k_1} \cdots x_n^{k_n}$ in the graph polynomial
of $G$, there exists $i$ such that $f(v_i) \le k_i$; in particular, we have
$f(v_i)\le j_i\le k_i$, where $i$ is chosen to show that $G-e$ is not $f$-AT. 
Hence, $G$ is not $f$-AT, a contradiction.
\end{proof}

Let $K_{2*t}$ denote the complete multipartite graph with $t$ parts of size 2.
Both \cite{hefetz2011two} and \cite{ZhuEtAl} showed that $K_{2*t}$ is $f$-AT
when $f(v)=t$ for all $v$.
So a direct application of Lemma~\ref{subgraphLemma} yields the following.

\begin{cor}
If $G\subseteq K_{2*t}$, then $G$ is $f$-AT when $f(v)=t$ for all $v$.  So, if
$G$ is BK-free, then $G\not\subseteq K_{2*(\Delta(G)-1)}$.
\label{cor1}
\end{cor}

We need a refinement of Corollary~\ref{cor1} that works for
$G\subseteq\join{K_s}{K_{2*t}}$ when some of the lists are smaller than size
$s+t$.  The idea used to prove Theorem~\ref{AT-thm} was
generalized~\cite{Alon99,AlonEtAl96} to what is now called the Combinatorial
Nullstellensatz.  Schauz \cite{schauz2008algebraically} further sharpened this
result, by proving the following coefficient formula.  Versions of this
sharper result were also proved by Hefetz \cite{hefetz2011two} and Laso{\'n}
\cite{lason2010generalization}. Our presentation follows Laso{\'n}.  
%We only state the lemma for homogeneous polynomials, since that is all we need
%in the application to the graph polynomial.
Recall that for a polynomial $g\in\mathbb{F}[x_1,\ldots,x_n]$, we write
$g_{k_1,\ldots,k_n}$ for the coefficient in $g$ of $x_1^{k_1}\cdots x_n^{k_n}$.

\begin{lem}[Schauz
\cite{schauz2008algebraically}]\label{nullCoefficient}
Suppose $g \in \mathbb{F}[x_1, \ldots, x_n]$ and $k_1, \ldots, k_n \in \IN$
with $\sum_{i \in \irange{n}} k_i = \deg(g)$.  For any $C_1, \ldots, C_n
\subseteq \mathbb{F}$ with $\card{C_i} = k_i + 1$, we have
\begin{align}
g_{k_1, \ldots, k_n} = \sum_{(c_1, \ldots, c_n) \in C_1 \times \cdots \times
C_n} \frac{f(c_1, \ldots, c_n)}{N(c_1, \ldots, c_n)},
\label{coeffSum}
\end{align} where 
\begin{align}
N(c_1, \ldots, c_n) \DefinedAs \prod_{i \in \irange{n}} \prod_{d \in C_i - c_i}
(c_i - d).
\label{coeffProd}
\end{align}
\end{lem}

%\RED{Need some words explaining this lemma.  What is the definition of $f_{k_1,
%\ldots, k_n}$?  Is this the coefficient on the term $x_1^{k_1}\cdots x_n^{k_n}$?  
%Intuitively, what is $N(c_1,\ldots,c_n)$?}

Now we use Lemma~\ref{nullCoefficient} to prove the desired strengthening of
Corollary~\ref{cor1}.

\begin{lem}\label{CliqueJoinE2Power}
Let $G = \join{K_s}{K_{2*t}}$, let $A$ be an $(s+t)$-clique in $G$, and let $B =
V(G) \setminus A$.  Now $G$ is $f$-AT whenever $f(v) \ge s + t$ for all $v \in A$
and $f(v) \ge t$ for all $v \in B$.
\end{lem}
\begin{proof}
Let $r = s+t$.  Say $A = \set{a_1, \ldots, a_r}$ and $B = \set{b_1, \ldots, b_t}$. 
Let $g$ be the graph polynomial of $G$, with the vertex order $a_1,b_1,
 \ldots, a_t, b_t, a_{t+1}, \ldots,a_r$. 
Recall that it suffices to show that $g_{k_1,\ldots,k_n}\ne 0$ for some choice
of $k_1,\ldots,k_n$ with $f(v_i)>k_i$ for all $i$.
To do this, we apply Lemma \ref{nullCoefficient}.  

Now $\deg(g) = |E(G)| = \binom{r}{2} + \binom{t}{2} + t(r-1)$.
For each $i \in \irange{t}$, let $L(a_i) = \irange{r}$ and $L(b_i) = \irange{t}$. 
For each $i \in \irange{r} \setminus \irange{t}$, let $L(a_i) = \irange{i}$.  Note
that $\sum_{i \in \irange{r}} \parens{|L(a_i)| - 1} + \sum_{i \in \irange{t}}
\parens{|L(b_i)| - 1} = t(r-1)+(r-t)(t+r-1)/2+t(t-1)=|E(G)|$, so these lists
will work for the $A_j$ in Lemma \ref{nullCoefficient}.  Also note that
$|L(a_i)| \le f(a_i)$ for all $i \in \irange{r}$ and $|L(b_i)| \le f(b_i)$ for
all $i \in \irange{t}$, so showing that the corresponding coefficient of $g$ is
nonzero will prove the lemma.

The sum in \eqref{coeffSum} of Lemma \ref{nullCoefficient} is zero at every term
that is not a proper coloring of $G$ from $L$.  By construction, all proper
colorings of $G$ from $L$ must assign $1,\ldots, t$ to vertices of $B$.  For
each vertex $b_i$ of $B$, its only non-neighbor is $a_i$.  Hence in every
proper coloring of $G$ from $L$, each of colors $1,\ldots, t$ is assigned to
some pair $(a_i,b_i)$.  As a result, vertex $a_{t+1}$ must get color $t+1$,
vertex $a_{t+2}$ must get color $t+2$, etc.  More precisely, for every $i\in
\{t+1,\ldots,r\}$, vertex $a_i$ gets color $i$.
Said differently, any
coloring of $G$ from $L$ can be obtained from any given such coloring by
permuting $1,\ldots, t$.  
Thus, the function $N$ in \eqref{coeffProd} of Lemma~\ref{nullCoefficient} gives
the same nonzero value on all such colorings (since for all $i,j\in[t]$, we
have $L(a_i)=L(a_j)$ and $L(b_i)=L(b_j)$).

The previous paragraph implies that the sum in \eqref{coeffSum} of Lemma
\ref{nullCoefficient} is a nonzero constant multiplied by the sum of the graph
polynomial $g'$ of $G[a_1,
\ldots, a_t, b_1, \ldots, b_t]$ evaluated at some points where $a_i$ and $b_i$
get the same value for each $i \in \irange{t}$. Any such evaluation is the
fourth power of an integer, since edges $a_ia_j, a_ib_j, b_ia_j, b_ib_j$ each
contribute the same factor.  In particular, all the terms in the sum have the
same sign.  Hence, by Lemma \ref{nullCoefficient}, the coefficient in question
is nonzero, so $G$ is $f$-AT.
\end{proof}

\begin{lem}
\label{ATPerfect}
Let $G$ be the complement of a bipartite graph with parts $A$ and $B$. If $f(v)
\ge \omega(G)$ for all $v \in A$ and $f(v) \ge |B|$ for all $v \in B$, then $G$
is $f$-AT.
\end{lem}
\begin{proof}
Define $G$ and $f$ as in the lemma.
We can assume that $\card{A}=\omega(G)$; so, in particular, $\card{A}\ge\card{B}$.  If not, then add
$\omega(G)-\card{A}$ vertices to $A$ that are adjacent only to vertices in $A$
(this does not increase $f(v)$ for any $v$).  So we have $\card{B}\le
\card{A}=\omega(G)$.  For each $S\subseteq A\cup B$, let
$\overline{N}(S)$ denote the neighbors of at least one vertex of
$S$ in $\overline{G}$, the complement of $G$.  Since $\omega(G)=|A|$, for each
$S\subseteq B$, we have $|\overline{N}(S)|\ge |S|$; otherwise $A\cup S\setminus
\overline{N}(S)$ is a clique of $G$ bigger than $A$.
So Hall's Theorem implies that $G \subseteq \join{K_{|A| - |B|}}{K_{2*|B|}}$.
%if not, consider a $S\subseteq B$ such that $|N_A(S)|<|S|$, and
%observe that than $A$).  
Hence, by Lemma \ref{CliqueJoinE2Power} and Lemma \ref{subgraphLemma}, $G$ is $f$-AT.
\end{proof}

\section{Reduction from Quasi-line Graphs to Line Graphs}
In this section, we prove that Theorem~\ref{MainThm} is true (for quasi-line
graphs) if it is true for line graphs.  Recall our general approach, based on
the quasi-line structure theorem, given in Theorem~\ref{QuasilineStructure}.
We assume that Theorem~\ref{MainThm} is false, and choose $G$ to be a minimal
counterexample; thus, $G$ is BK-free.
In Section~\ref{circular-interval-graphs}, we prove that $G$ is not a
circular interval graph.  In Section~\ref{homogeneous-pairs}, we prove that $G$
has no non-linear homogeneous pair of cliques.  Finally,
in Section~\ref{2joins}, we consider when $G$ is a composition of linear
interval strips (which include line graphs, as a special case).  
We reduce this case to the case of line graphs, which we handle in
Section~\ref{line-graphs}.

\subsection{Handling circular interval graphs}
\label{circular-interval-graphs}
The following proof is nearly identical to the one we gave
in~\cite{BK-claw-free} for the list-coloring analogue, but we reproduce it here
for completeness.  One notable difference is that all of the list-coloring
lemmas used to show reducibility in that proof have been replaced here by the
Alon--Tarsi orientations in Figure~\ref{ATpics1}.

\begin{lem}\label{NotCircularIntervalIfBKCritical}
Let $G$ be a BK-free graph with $\omega(G)<\Delta(G)$. 
If $G$ is a circular interval graph, then $\Delta(G) < 9$.
%$\AT(G)\le \Delta(G)-1$.
\end{lem}
\begin{proof}
Suppose to the contrary that $G$ is a circular interval graph that is
BK-free, has $\Delta(G) \ge 9$, and does not contain $K_{\Delta(G)}$.
Let $K$ be a maximum clique in $G$.  By symmetry we may assume that
$V(K)=\{v_1,v_2,\ldots,v_t\}$ for some $t\le \Delta-1$; further, if possible we
label the vertices so that $v_{t-3}\adj v_{t+1}$ and the edge goes through
$v_{t-2},v_{t-1},v_t$.

\claim{1} {$v_1\nonadj v_{t+1}$ and $v_2\nonadj v_{t+2}$ and
$v_1\nonadj v_{t+2}$.}
Assume the contrary.
Clearly we cannot have $v_1\adj v_{t+1}$ and have the edge go through
$v_2,v_3,\ldots, v_t$ (since then we get a clique of size $t+1$).
Similarly, we cannot have $v_2\adj v_{t+2}$ and have the edge go through
$v_3,v_4,\ldots,v_{t+1}$.  So assume the edge $v_1v_{t+2}$ exists and 
goes around the other way.
If $v_1\adj v_{t+1}$, then let $G'=G\setminus \{v_1\}$ and if $v_1\nonadj
v_{t+1}$, then let $G'=G\setminus \{v_1,v_{t+1}\}$.  Now let
$V_1=\{v_2,v_3,\ldots,v_t\}$ and $V_2=V(G')\setminus V_1$.  Let $K'=G[V_1]$ and
$L'=G[V_2]$; note that $K'$ and $L'$ are each cliques of size at most
$\Delta-2$.  Now for each $S\subseteq V_2$,
we have $|N_{\overline{G}}(S)\cap V_1|\ge |S|$
(otherwise we get a clique of size $t$ in $G'$ and a clique of size $t+1$ in $G$).  Now by Hall's Theorem, we have a matching in $\overline{G}$ between $V_1$
and $V_2$ that saturates $V_2$.  
%This implies that $G'\subseteq E_2^{\Delta-2}$,
%which in turn gives $G\subseteq E_2^{\Delta-1}$.  
This implies that $G'\subseteq K_{2*(\Delta-2)}$,
which in turn gives $G\subseteq K_{2*(\Delta-1)}$.  
%\RED{Fix this to use BK-free: By Lemma~\ref{E2n}, $G$ is
%$(\Delta-1)$-choosable, which is a contradiction.
%Does the following fix look good?
This contradicts Corollary~\ref{cor1}.
%So Corollary~\ref{cor1} implies that $\AT(G)\le \Delta-1$.
%}

\begin{figure}[ht]
\subfloat[\label{K4-e}: EE=2, EO=1]{
\makebox[.33\textwidth]{
\begin{tikzpicture}[scale = 8]
\tikzstyle{VertexStyle}=[shape = circle, minimum size = 6pt, inner sep = 1.2pt, draw]
\Vertex[x = 0.45, y = 0.70, L = \small {$2$}]{v0}
\Vertex[x = 0.45, y = 0.50, L = \small {$1$}]{v1}
\Vertex[x = 0.60, y = 0.70, L = \small {$1$}]{v2}
\Vertex[x = 0.60, y = 0.50, L = \small {$1$}]{v3}
\Edge[style = {pre}](v0)(v2)
\Edge[style = {post}](v1)(v2)
\Edge[style = {post}](v0)(v3)
\Edge[style = {pre}](v1)(v3)
\Edge[style = {post}](v1)(v0)
\end{tikzpicture}}}
\subfloat[\label{K3vE2} EE = 4, EO = 3]{
\makebox[.33\textwidth]{
\begin{tikzpicture}[scale = 8]
\tikzstyle{VertexStyle}=[shape = circle, minimum size = 6pt, inner sep = 1.2pt, draw]
\Vertex[x = 0.45, y = 0.70, L = \small {$2$}]{v0}
\Vertex[x = 0.45, y = 0.50, L = \small {$3$}]{v1}
\Vertex[x = 0.60, y = 0.70, L = \small {$1$}]{v2}
\Vertex[x = 0.60, y = 0.50, L = \small {$2$}]{v3}
\Vertex[x = 0.35, y = 0.60, L = \small {$1$}]{v4}
\Edge[style = {pre}](v0)(v2)
\Edge[style = {pre}](v1)(v2)
\Edge[style = {post}](v0)(v3)
\Edge[style = {pre}](v1)(v3)
\Edge[style = {post}](v1)(v0)
\Edge[style = {pre}](v4)(v0)
\Edge[style = {post}](v4)(v1)
\Edge[style = {pre}](v2)(v4)
\Edge[style = {pre}](v3)(v4)
\end{tikzpicture}}}
\subfloat[\label{K2vAntichair} EE = 81, EO = 80]{
\makebox[.33\textwidth]{
\begin{tikzpicture}[scale = 8]
\tikzstyle{VertexStyle} = []
\tikzstyle{EdgeStyle} = []
\tikzstyle{labeledStyle}=[shape = circle, minimum size = 6pt, inner sep = 1.2pt, draw]
\tikzstyle{PreEdge}=[pre]
\tikzstyle{PostEdge}=[post]
\Vertex[style = labeledStyle, x = 0.20, y = 0.80, L = \small {2}]{v0}
\Vertex[style = labeledStyle, x = 0.30, y = 0.85, L = \small {2}]{v1}
\Vertex[style = labeledStyle, x = 0.30, y = 0.75, L = \small {3}]{v2}
\Vertex[style = labeledStyle, x = 0.40, y = 0.80, L = \small {2}]{v3}
\Vertex[style = labeledStyle, x = 0.50, y = 0.80, L = \small {2}]{v4}
\Vertex[style = labeledStyle, x = 0.25, y = 0.60, L = \small {2}]{v5}
\Vertex[style = labeledStyle, x = 0.45, y = 0.60, L = \small {4}]{v6}
\Edge[style = PreEdge, labelstyle={auto=right, fill=none}](v1)(v3)
\Edge[style = PreEdge, labelstyle={auto=right, fill=none}](v2)(v3)
\Edge[style = PostEdge, label = \small {}, labelstyle={auto=right, fill=none}](v1)(v0)
\Edge[style = PostEdge, label = \small {}, labelstyle={auto=right, fill=none}](v2)(v0)
\Edge[style = PreEdge, labelstyle={auto=right, fill=none}](v1)(v2)
\Edge[style = PreEdge, labelstyle={auto=right, fill=none}](v3)(v4)
\Edge[style = PreEdge, labelstyle={auto=right, fill=none}](v6)(v5)
\Edge[style = PostEdge, label = \small {}, labelstyle={auto=right, fill=none}](v0)(v6)
\Edge[style = PostEdge, label = \tiny {}, labelstyle={auto=right, fill=none}](v1)(v6)
\Edge[style = PreEdge, labelstyle={auto=right, fill=none}](v2)(v6)
\Edge[style = PostEdge, label = \tiny {}, labelstyle={auto=right, fill=none}](v3)(v6)
\Edge[style = PreEdge, labelstyle={auto=right, fill=none}](v4)(v6)
\Edge[style = PostEdge, label = \tiny {}, labelstyle={auto=right, fill=none}](v0)(v5)
\Edge[style = PostEdge, label = \tiny {}, labelstyle={auto=right, fill=none}](v1)(v5)
\Edge[style = PreEdge, labelstyle={auto=right, fill=none}](v2)(v5)
\Edge[style = PreEdge, labelstyle={auto=right, fill=none}](v3)(v5)
\Edge[style = PreEdge, labelstyle={auto=right, fill=none}](v4)(v5)
\end{tikzpicture}}}

\subfloat[\label{K4vE2} EE = 16, EO = 17]{
\makebox[.33\textwidth]{
\begin{tikzpicture}[scale = 8]
\tikzstyle{VertexStyle}=[shape = circle, minimum size = 6pt, inner sep = 1.2pt, draw]
\Vertex[x = 0.40, y = 0.85, L = \small {$2$}]{v0}
\Vertex[x = 0.55, y = 0.85, L = \small {$3$}]{v1}
\Vertex[x = 0.40, y = 0.65, L = \small {$4$}]{v2}
\Vertex[x = 0.55, y = 0.65, L = \small {$1$}]{v3}
\Vertex[x = 0.70, y = 0.90, L = \small {$2$}]{v4}
\Vertex[x = 0.70, y = 0.60, L = \small {$2$}]{v5}
\Edge[style = {post}](v0)(v4)
\Edge[style = {pre}](v1)(v4)
\Edge[style = {pre}](v2)(v4)
\Edge[style = {post}](v3)(v4)
\Edge[style = {post}](v0)(v5)
\Edge[style = {pre}](v1)(v5)
\Edge[style = {pre}](v2)(v5)
\Edge[style = {post}](v3)(v5)
\Edge[style = {post}](v0)(v3)
\Edge[style = {pre}](v1)(v3)
\Edge[style = {pre}](v2)(v3)
\Edge[style = {pre}](v0)(v2)
\Edge[style = {post}](v1)(v2)
\Edge[style = {pre}](v0)(v1)
\end{tikzpicture}}}
\subfloat[\label{K4v2E2a} EE = 512, EO = 515]{
\makebox[.33\textwidth]{
\begin{tikzpicture}[scale = 8]
\tikzstyle{VertexStyle}=[shape = circle, minimum size = 6pt, inner sep = 1.2pt,
draw]
\Vertex[x = 0.45, y = 0.70, L = \small {$2$}]{v0}
\Vertex[x = 0.30, y = 0.70, L = \small {$3$}]{v1}
\Vertex[x = 0.45, y = 0.50, L = \small {$5$}]{v2}
\Vertex[x = 0.65, y = 0.75, L = \small {$2$}]{v3}
\Vertex[x = 0.65, y = 0.65, L = \small {$2$}]{v4}
\Vertex[x = 0.65, y = 0.55, L = \small {$2$}]{v5}
\Vertex[x = 0.65, y = 0.45, L = \small {$2$}]{v6}
\Vertex[x = 0.30, y = 0.50, L = \small {$4$}]{v7}
\Edge[style = {post}](v1)(v0)
\Edge[style = {post}](v2)(v0)
\Edge[style = {post}](v2)(v1)
\Edge[style = {post}](v0)(v3)
\Edge[style = {pre}](v1)(v3)
\Edge[style = {pre}](v2)(v3)
\Edge[style = {post}](v0)(v4)
\Edge[style = {post}](v1)(v4)
\Edge[style = {pre}](v2)(v4)
\Edge[style = {post}](v0)(v5)
\Edge[style = {pre}](v1)(v5)
\Edge[style = {pre}](v2)(v5)
\Edge[style = {post}](v0)(v6)
\Edge[style = {post}](v1)(v6)
\Edge[style = {pre}](v2)(v6)
\Edge[style = {pre}](v7)(v1)
\Edge[style = {post}](v7)(v2)
\Edge[style = {pre}](v7)(v0)
\Edge[style = {pre}](v3)(v7)
\Edge[style = {post}](v4)(v7)
\Edge[style = {pre}](v5)(v7)
\Edge[style = {post}](v6)(v7)
\end{tikzpicture}}}
\subfloat[\label{K4v2E2b} EE = 751, EO = 750]{
\makebox[.33\textwidth]{
\begin{tikzpicture}[scale = 8]
\tikzstyle{VertexStyle}=[shape = circle, minimum size = 6pt, inner sep = 1.2pt, draw]
\Vertex[x = 0.45, y = 0.70, L = \small {$2$}]{v0}
\Vertex[x = 0.30, y = 0.70, L = \small {$3$}]{v1}
\Vertex[x = 0.45, y = 0.50, L = \small {$4$}]{v2}
\Vertex[x = 0.65, y = 0.75, L = \small {$2$}]{v3}
\Vertex[x = 0.65, y = 0.65, L = \small {$3$}]{v4}
\Vertex[x = 0.65, y = 0.55, L = \small {$2$}]{v5}
\Vertex[x = 0.65, y = 0.45, L = \small {$2$}]{v6}
\Vertex[x = 0.30, y = 0.50, L = \small {$5$}]{v7}
\Edge[style = {post}](v1)(v0)
\Edge[style = {post}](v2)(v0)
\Edge[style = {post}](v2)(v1)
\Edge[style = {post}](v0)(v3)
\Edge[style = {pre}](v1)(v3)
\Edge[style = {pre}](v2)(v3)
\Edge[style = {post}](v0)(v4)
\Edge[style = {post}](v1)(v4)
\Edge[style = {pre}](v2)(v4)
\Edge[style = {post}](v0)(v5)
\Edge[style = {pre}](v1)(v5)
\Edge[style = {post}](v2)(v5)
\Edge[style = {post}](v0)(v6)
\Edge[style = {post}](v1)(v6)
\Edge[style = {pre}](v2)(v6)
\Edge[style = {pre}](v7)(v1)
\Edge[style = {post}](v7)(v2)
\Edge[style = {pre}](v7)(v0)
\Edge[style = {pre}](v3)(v7)
\Edge[style = {post}](v4)(v7)
\Edge[style = {post}](v5)(v7)
\Edge[style = {post}](v6)(v7)
\Edge[style = {post}](v3)(v4)
\end{tikzpicture}}}

\subfloat[\label{K4v2E2c} EE = 1097, EO = 1096]{
\makebox[.35\textwidth]{
\begin{tikzpicture}[scale = 8]
\tikzstyle{VertexStyle}=[shape = circle, minimum size = 6pt, inner sep = 1.2pt, draw]
\Vertex[x = 0.45, y = 0.70, L = \small {$2$}]{v0}
\Vertex[x = 0.30, y = 0.70, L = \small {$3$}]{v1}
\Vertex[x = 0.45, y = 0.50, L = \small {$4$}]{v2}
\Vertex[x = 0.65, y = 0.75, L = \small {$3$}]{v3}
\Vertex[x = 0.65, y = 0.65, L = \small {$2$}]{v4}
\Vertex[x = 0.65, y = 0.55, L = \small {$2$}]{v5}
\Vertex[x = 0.65, y = 0.45, L = \small {$3$}]{v6}
\Vertex[x = 0.30, y = 0.50, L = \small {$5$}]{v7}
\Edge[style = {post}](v1)(v0)
\Edge[style = {post}](v2)(v0)
\Edge[style = {post}](v2)(v1)
\Edge[style = {post}](v0)(v3)
\Edge[style = {pre}](v1)(v3)
\Edge[style = {pre}](v2)(v3)
\Edge[style = {post}](v0)(v4)
\Edge[style = {pre}](v1)(v4)
\Edge[style = {post}](v2)(v4)
\Edge[style = {post}](v0)(v5)
\Edge[style = {post}](v1)(v5)
\Edge[style = {pre}](v2)(v5)
\Edge[style = {post}](v0)(v6)
\Edge[style = {post}](v1)(v6)
\Edge[style = {pre}](v2)(v6)
\Edge[style = {pre}](v7)(v1)
\Edge[style = {post}](v7)(v2)
\Edge[style = {pre}](v7)(v0)
\Edge[style = {pre}](v3)(v7)
\Edge[style = {post}](v4)(v7)
\Edge[style = {post}](v5)(v7)
\Edge[style = {post}](v6)(v7)
\Edge[style = {pre}](v3)(v4)
\Edge[style = {post}](v5)(v6)
\end{tikzpicture}}}
\subfloat[\label{K3vP4} EE = 108, EO = 107]{
\makebox[.33\textwidth]{
\begin{tikzpicture}[scale = 8]
\tikzstyle{VertexStyle}=[shape = circle, minimum size = 6pt, inner sep = 1.2pt, draw]
\Vertex[x = 0.45, y = 0.70, L = \small {$2$}]{v0}
\Vertex[x = 0.35, y = 0.60, L = \small {$3$}]{v1}
\Vertex[x = 0.45, y = 0.50, L = \small {$4$}]{v2}
\Vertex[x = 0.65, y = 0.75, L = \small {$2$}]{v3}
\Vertex[x = 0.65, y = 0.65, L = \small {$2$}]{v4}
\Vertex[x = 0.65, y = 0.55, L = \small {$2$}]{v5}
\Vertex[x = 0.65, y = 0.45, L = \small {$3$}]{v6}
\Edge[style = {post}](v1)(v0)
\Edge[style = {post}](v2)(v0)
\Edge[style = {post}](v2)(v1)
\Edge[style = {post}](v4)(v5)
\Edge[style = {pre}](v6)(v5)
\Edge[style = {post}](v0)(v3)
\Edge[style = {pre}](v1)(v3)
\Edge[style = {pre}](v2)(v3)
\Edge[style = {post}](v0)(v4)
\Edge[style = {post}](v1)(v4)
\Edge[style = {pre}](v2)(v4)
\Edge[style = {post}](v0)(v5)
\Edge[style = {pre}](v1)(v5)
\Edge[style = {pre}](v2)(v5)
\Edge[style = {post}](v0)(v6)
\Edge[style = {post}](v1)(v6)
\Edge[style = {pre}](v2)(v6)
\Edge[style = {pre}](v3)(v4)
\end{tikzpicture}}}
\subfloat[\label{K2vC4}: EE = 30, EO = 28]{
\makebox[.33\textwidth]{
\begin{tikzpicture}[scale = 8]
\tikzstyle{VertexStyle}=[shape = circle, minimum size = 6pt, inner sep = 1.2pt, draw]
\Vertex[x = 0.45, y = 0.70, L = \small {$2$}]{v0}
\Vertex[x = 0.45, y = 0.50, L = \small {$3$}]{v1}
\Vertex[x = 0.65, y = 0.75, L = \small {$2$}]{v2}
\Vertex[x = 0.65, y = 0.45, L = \small {$2$}]{v3}
\Vertex[x = 0.85, y = 0.75, L = \small {$2$}]{v4}
\Vertex[x = 0.85, y = 0.45, L = \small {$2$}]{v5}
\Edge[style = {pre}](v0)(v2)
\Edge[style = {pre}](v1)(v2)
\Edge[style = {post}](v0)(v3)
\Edge[style = {pre}](v1)(v3)
\Edge[style = {post}](v1)(v0)
\Edge[style = {pre}](v5)(v4)
\Edge[style = {post}](v5)(v3)
\Edge[style = {post}](v4)(v2)
\Edge[style = {pre}](v2)(v3)
\Edge[style = {post}](v0)(v4)
\Edge[style = {post}](v1)(v4)
\Edge[style = {post}](v0)(v5)
\Edge[style = {pre}](v1)(v5)
\end{tikzpicture}}}

\caption{Subgraphs forbidden by Alon-Tarsi orientations, used in
Lemma~\ref{NotCircularIntervalIfBKCritical}.
\label{ATpics1}}
\end{figure}
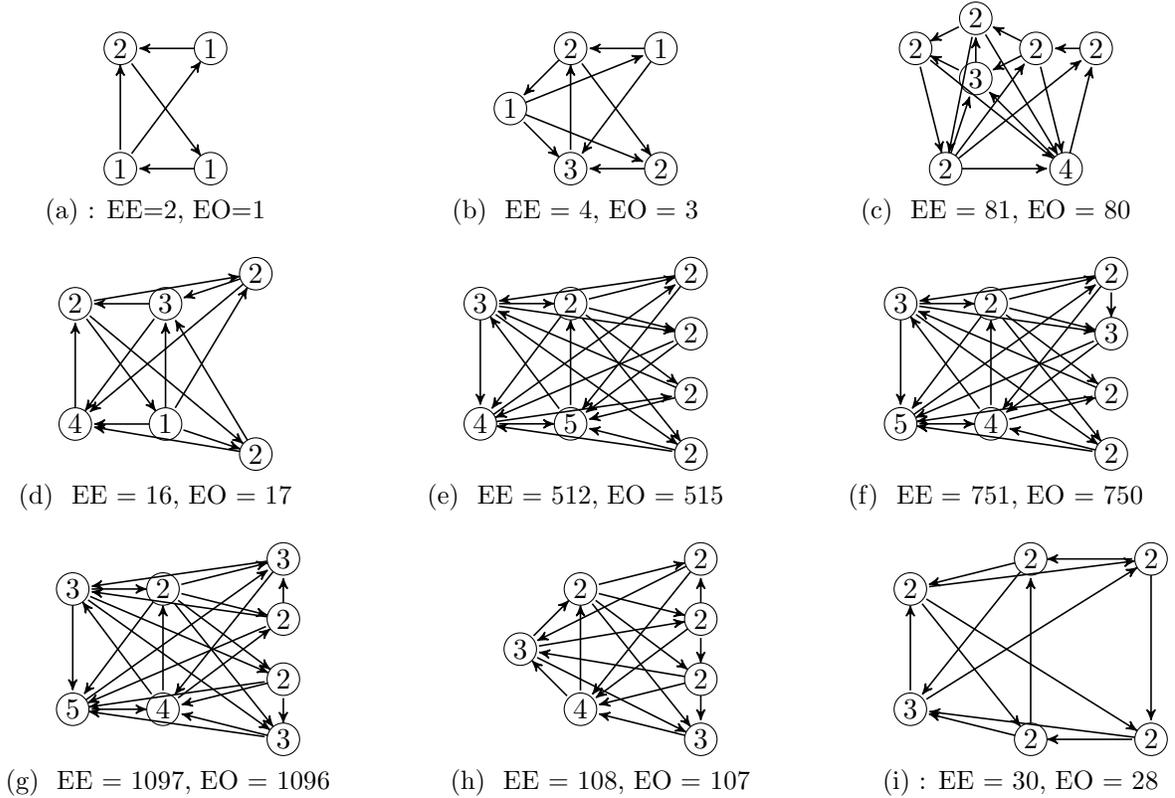

\claim{2} {$v_{t-3}\adj v_{t+1}$ and the edge passes through
$v_{t-2},v_{t-1},v_t$.}
Assume the contrary.
Since $t\le \Delta-1$ and $\delta(G)\ge\Delta-1$, each vertex in $K$ has a
neighbor outside of $K$; in particular, $v_4$ has some neighbor outside of $K$.
If $t\ge 7$, then by (reflectional) symmetry we could have labeled the
vertices so that $v_{t-3}\adj v_{t+1}$ (and the edge passes through
$v_{t-2},v_{t-1},v_t$).  So we must have $t\le 6$.
Each vertex $v$ that is high has either at least $\ceil{\Delta/2}$ clockwise
neighbors or at least $\ceil{\Delta/2}$ counterclockwise neighbors.  
This gives a clique of size $1+\ceil{\Delta/2}\ge 6$.  Thus, $t=6$ and
$v_{t-3}=v_3$.
If $v_3$ is high, then either $v_3$ has at
least 4 clockwise neighbors, so $v_3\adj v_7$, or else $v_3$ has at least 6
counterclockwise neighbors, so $|K|\ge 7$.  Thus, we may assume that $v_3$ is
low; by symmetry (and our choice of labeling prior to Claim~1) $v_4$ is also
low.  Now since $v_4$ has only 3 counterclockwise neighbors, we get $v_4\adj
v_7$ (in fact, we get $v_4\adj v_9$).  Thus, $\{v_3,v_4,v_5,v_6,v_7\}$ induces
$\join{K_3}{E_2}$ with a low degree vertex in both the $K_3$ and the $E_2$, which 
is $f$-AT, as shown in Figure~\ref{ATpics1}\subref{K3vE2}.

\claim{3} {$v_{t-2} \nonadj v_{t+2}$ and $v_{t-1}\nonadj v_{t+2}$.}
First, assume to the contrary that $v_{t-2}\adj v_{t+2}$.  By
Claim~1 the edge must go through $v_{t-1},v_t,v_{t+1}$.  If $v_{t-3}\adj v_{t+2}$,
then the set $\{v_1,v_2,v_{t-3},v_{t-2},v_{t-1}$, $v_t,v_{t+1},v_{t+2}\}$ induces
$\join{K_4}{B}$, where $B$ is not almost complete; this subgraph is $f$-AT,
as shown in Figure~\ref{ATpics1}\subref{K4v2E2a}--\subref{K2vC4}.
If $v_{t-3}\nonadj v_{t+2}$, then the set $\{v_1, v_{t-3}, v_{t-2}, v_{t-1}$,
$v_t, v_{t+1}, v_{t+2}\}$ induces $\join{K_3}{P_4}$ is $f$-AT, 
as shown in Figure~\ref{ATpics1}\subref{K3vP4}.
Hence, $v_{t-2}\nonadj v_{t+2}$.

So assume that $v_{t-1}\adj v_{t+2}$.  Now
$\{v_1,v_{t-3},v_{t-2},v_{t-1},v_t,v_{t+1},v_{t+2}\}$ induces
$\join{K_2}{\mbox{antichair}}$ (with $v_{t-1},v_t$ in the $K_2$), 
which is $f$-AT, as shown in Figure~\ref{ATpics1}\subref{K2vAntichair}.

%\claim{4} {$\AT(G)\le \Delta(G)-1$.}
\claim{4} {The lemma is true.}
Let $S=\{v_{t-3},v_{t-2},v_{t-1},v_t\}$.  If any vertex of $S$ is low, then
$S\cup\{v_1,v_{t+1}\}$ induces $\join{K_4}{E_2}$ with a low vertex in the $K_4$, 
which is $f$-AT, as shown in Figure~\ref{ATpics1}\subref{K4vE2}.
So all of $S$ is high.  If
$v_t\nonadj v_{t+2}$, then $\{v_t,v_{t-1},\ldots, v_{t-\Delta+1}\}$ (subscripts
are modulo $n$) induces $K_{\Delta}$.  So $v_t\adj v_{t+2}$.  Since
$v_{t-1}\nonadj v_{t+2}$ and all of $S$ is high, there exists a vertex
$v_n$ that is not adjacent to $v_t$ but is adjacent to the rest of $S$.
Formally, $v_n\in (\cap_{v\in (S\setminus\{v_t\})}N(v))\setminus N(v_t)$.  
Clearly the edge from $v_{t-1}$ to $v_n$ must go through $v_{t-2}$.
Since $v_n\nonadj v_t$, we have $n<1$.  However, if $n<0$, then $G$ contains a
clique larger than $K$.  Thus, we may assume $v_n=v_0$.

%\RED{Fix to use BK-free: Now we must have $v_n\nonadj
%v_{t+1}$ (for otherwise $G$ is $(\Delta-1)$-choosable, as in Claim~1).
%This fix seems easy.  }
If $v_n\adj v_{t+1}$, then $G$ can be covered by two
cliques: $K$ and $G\setminus K$.  As in Claim 1, we show that $\overline{G}$ has
a matching between $K$ and $G\setminus K$ that saturates $G\setminus K$.  Thus,
$G\subseteq K_{2*(\Delta-1)}$, which contradicts Corollary~\ref{cor1}.
Since $v_n\nonadj v_{t+1}$, we get $\join{K_3}{P_4}$ induced by
$\{v_{t+1},v_t,v_{t-1},v_{t-2},v_{t-3},v_1,v_n\}$.  Again, this subgraph is
$f$-AT, as shown in Figure~\ref{ATpics1}\subref{K3vP4}.
\end{proof}

\subsection{Handling non-linear homogeneous pairs of cliques}
\label{homogeneous-pairs}
%\begin{lem}
%\label{ATPerfect}
%Let $G$ be the complement of a bipartite graph with parts $A$ and $B$. If $f(v)
%\ge \omega(G)$ for $v \in A$ and $f(v) \ge |B|$ for $v \in B$, then $G$ is
%$f$-AT.
%\end{lem}

%The statement of the following lemma is similar to one 
%in~\cite{claw-free}, but the proof here differs significantly, since many of
%the tools we used there are not available for Alon--Tarsi orientations.

\begin{lem}\label{NoNonLinear}
%If $G$ is a BK-free graph with $\Delta(G) \ge 9$ not containing $K_{\Delta(G)}$, then $G$ contains no non-linear homogeneous pair of cliques.
%
Let $G$ be a BK-free graph with $\omega(G)<\Delta(G)$. 
If $G$ has a non-linear homogeneous pair of cliques, then $\Delta(G)<9$.
%If $G$ has a non-linear homogeneous pair of cliques, then $\AT(G)\le \Delta(G)-1$.
\end{lem}
\begin{proof}
Suppose to the contrary that $G$ is a BK-free graph with $\Delta(G) \ge 9$
and $\omega(G)<\Delta(G)$ and that $G$ contains a non-linear homogeneous pair of
cliques $(A,B)$.  Let $H \DefinedAs G[A \cup B]$ and let $f_H(v) \DefinedAs
d_H(v) - 1 + \Delta(G) - d_G(v)$ for all $v \in V(H)$.  Now $H$ is not
complete, since it it is non-linear and hence induces a $C_4$.
Our general approach is to show that $G$ contains some induced $f$-AT subgraph
in Figure~\ref{ATpics2}, where $f(v)=d(v)$
when $v$ is low and $f(v)=d(v)-1$ otherwise.

Note that $f_H(v) = d_H(v) - 1$ if $v$ is high and $f_H(v) = d_H(v)$ if $v$ is
low.  For each $X \in \set{A,B}$, let $\delta_X \DefinedAs \min_{v \in X}
d_H(v)$ and $\Delta_X \DefinedAs \max_{v \in X} d_H(v)$.  Since each vertex in
$G$ has degree either $\Delta(G)$ or $\Delta(G)-1$, and $(A,B)$ is a
homogeneous pair of cliques, we have $\Delta_X \le \delta_X + 1$ for each $X \in
\set{A,B}$;  equality holds when $X$ contains both a high and a low vertex.
Also, $f_H(v) = \Delta_X - 1$ for each $v \in X$ whenever $X$ contains a high
vertex.  Let $W$ be an arbitrary maximum clique in $H$.
%
%\smallskip

\claim{0} {For $\set{X,Y} = \set{A,B}$, either $\Delta_X \le |W|$ or $\Delta_Y \le |Y|$.}
%
%\begin{proof}
Since $G$ is BK-free, $f_H$ cannot satisfy the hypotheses of
Lemma~\ref{ATPerfect}.  Unpacking what that means gives precisely $\Delta_X \le
|W|$ or $\Delta_Y \le |Y|$.
%\end{proof}

\begin{figure}[bht]
\centering
\subfloat[3 and 3]{
\begin{tikzpicture}[scale = 8]
\tikzstyle{VertexStyle}=[shape = circle, minimum size = 6pt, inner sep = 1.2pt,
                                 draw]
\Vertex[x = 0.45, y = 0.80, L = \small {}]{v0}
\Vertex[x = 0.65, y = 0.80, L = \small {}]{v1}
\Vertex[x = 0.45, y = 0.60, L = \small {}]{v2}
\Vertex[x = 0.65, y = 0.60, L = \small {}]{v3}
\Vertex[x = 0.40, y = 0.45, L = \small {$u$}]{v4}
\Vertex[x = 0.70, y = 0.45, L = \small {$z$}]{v5}
\Edge[](v0)(v3)
\Edge[](v0)(v4)
\Edge[](v1)(v0)
\Edge[](v1)(v2)
\Edge[](v2)(v0)
\Edge[](v3)(v1)
\Edge[](v3)(v2)
\Edge[](v4)(v2)
\Edge[](v4)(v5)
\Edge[](v5)(v1)
\Edge[](v5)(v3)
\end{tikzpicture}
\label{fig:AandBAre3a}
% pic1a
}
\subfloat[3 and 2]{
\begin{tikzpicture}[scale = 8]
\tikzstyle{VertexStyle}=[shape = circle, minimum size = 6pt, inner sep = 1.2pt, draw]
\Vertex[x = 0.60, y = 0.80, L = \small {}]{v0}
\Vertex[x = 0.60, y = 0.60, L = \small {}]{v1}
\Vertex[x = 0.80, y = 0.60, L = \small {}]{v2}
\Vertex[x = 0.55, y = 0.45, L = \small {$u$}]{v3}
\Vertex[x = 0.85, y = 0.45, L = \small {$z$}]{v4}
\Edge[](v0)(v2)
\Edge[](v0)(v3)
\Edge[](v1)(v0)
\Edge[](v2)(v1)
\Edge[](v3)(v1)
\Edge[](v3)(v4)
\Edge[](v4)(v2)
\end{tikzpicture}
\label{fig:AandBAre3b}
% pic1b
}
\subfloat[2 and 2]{
\begin{tikzpicture}[scale = 8]
\tikzstyle{VertexStyle}=[shape = circle, minimum size = 6pt, inner sep = 1.2pt, draw]
\Vertex[x = 0.60, y = 0.60, L = \tiny {}]{v0}
\Vertex[x = 0.80, y = 0.60, L = \tiny {}]{v1}
\Vertex[x = 0.55, y = 0.45, L = \small {$u$}]{v2}
\Vertex[x = 0.85, y = 0.45, L = \small {$z$}]{v3}
\Edge[](v1)(v0)
\Edge[](v2)(v0)
\Edge[](v2)(v3)
\Edge[](v3)(v1)
\end{tikzpicture}
\label{fig:AandBAre3c}
% pic1c
}
\caption{Possible non-linear homogeneous pair of cliques.}
\label{fig:AandBAre3}
\end{figure}
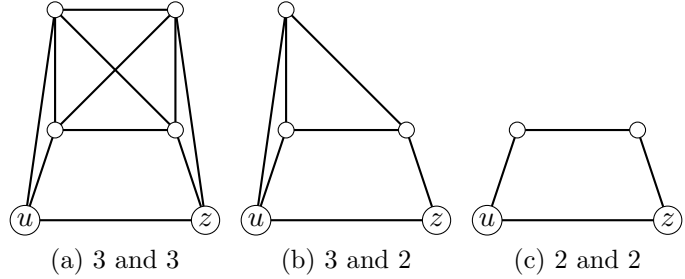

\claim{1} {Either $\card{W \cap A} \le 1$ or $\card{W \cap B} \le 1$.}
%
%\begin{proof}
Suppose instead that $\card{W \cap X} \ge 2$ for all $X \in \set{A,B}$.  Now
$\Delta_X\ge |X|+1$ for all $X \in \set{A,B}$, so applying Claim 0 gives
$\Delta_X \le |W|$ for all $X \in \set{A,B}$.  For all $v \in W \cap X$, we have
$d_H(v) \ge |W| - 1 + |X \setminus W|$.  Hence $|W| \ge \Delta_X \ge |W| - 1 +
|X \setminus W|$, which gives $|X \cap W| \ge |X| - 1$ for all $X \in \set{A,B}$.

Now we show that {$A \subseteq W$ or $B \subseteq W$.} Suppose instead that
$\card{W \cap A} = |A| - 1$ and $\card{W \cap B} = |B| - 1$.  Let $\set{u} = A
\setminus W$ and $\set{z} = B\setminus W$.  If $u$ has a neighbor $v\in W \cap
B$, then $d_H(v)\ge |W|+1$, a contradiction. Similarly, $z$ has no neighbors in
$W \cap A$.  So $d_A(z)\le 1$ and $d_B(u) \le 1$, which implies that $\Delta_A
\le |A| + 1$ and $\Delta_B \le |B| + 1$. Now $|B|+1\ge\Delta_B \ge |B|-1+|W\cap
A|$, so $|W\cap A|\le 2$.  By assumption $|W\cap A| = |A|-1$, so 
$|A|= |W\cap A|+1\le 3$; similarly, $|B|\le 3$.

Recall that $|W\cap A|\ge 2$, $|W\cap B|\ge 2$, $|A|\le 3$, and $|B|\le 3$.
Now our assumption that $\card{W \cap A} = |A| - 1$ and $\card{W \cap B} = |B|
- 1$  gives $|A| = |B| = 3$.  So $H$ must be as in
Figure~\ref{fig:AandBAre3}\subref{fig:AandBAre3a}.  Hence $d_G(u) = d_G(z) =
\Delta(G) - 1$.  But now either (i) some vertex outside $H$ is joined to just
one side of $H$ and $G$ contains the forbidden induced subgraph in Figure
\ref{ATpics2}\subref{fig:A3B3EulerA} or (ii) some vertex outside $H$ is joined
to both sides of $H$ and $G$ contains the graph in
Figure~\ref{ATpics2}\subref{fig:A3B3EulerB}.
 Each of these induced subgraphs is forbidden, which gives a contradiction. 
Thus, $A\subseteq W$ or $B\subseteq W$.

By symmetry, suppose $A\subseteq W$.  Now we get $\card{W}\ge
\card{A}+\card{B\cap W} \ge \card{A}+(\card{B}-1)$.
%Finally, we conclude that $\card{W}\ge \card{A}+\card{B}-1$.  
Thus, $H$ is almost complete, so it does not contain an induced $C_4$; this
contradicts the hypothesis of the lemma and so proves the claim.
%\end{proof}
%\smallskip

\claim{2}{Either $\omega(H) \le 2$ or the only possible maximum cliques in $H$
are $A$ and $B$.}
%
%\begin{proof}
Suppose to the contrary that $\omega(H) \ge 3$ and that $W$ is a maximum clique
in $H$ with $\card{W \cap A}\ge 1$ and ${\card{W \cap B} \ge 1}$.  By Claim 1
and symmetry, we may assume $\card{W \cap B} = 1$.  
Let $z_1, \ldots, z_t$ be the vertices of $B \setminus W$.  
Since $|W|=\omega(H)\ge 3$ and $|W\cap B|=1$, 
we must have $\card{W \cap A} \ge 2$ and hence $\card{W \cap A} \ge |A| - 1$,
as in the first paragraph of the proof of Claim~1.

So we have the two cases (i) $A\subseteq W$ and (ii) $|A\cap W|=|A|-1$.
First suppose that $A\subseteq W$.  
We begin with the case $|A|=2$.  Since $|W\cap B|=1$ and $3\le |W|\le
|A|+|W\cap B|\le 3$, we get that $|B|\le |W|\le 3$.  
We must have $|B|\ge 3$, since otherwise $H$ is almost complete, so it cannot
induce $C_4$.  So $|B|=3$, and $\Delta_B\ge |W| + 1$.  Since
$\Delta_B-\delta_B\le 1$, each of $z_1$ and $z_2$ has a neighbor in $A$; thus
$\Delta_A\ge |A| + 1$, which contradicts Claim 0.  Hence, $\card{A}\ge 3$.

\begin{figure}
\centering
\subfloat[E = 14, O = 12]{
\makebox[.33\textwidth]{
\begin{tikzpicture}[scale = 8]
\tikzstyle{VertexStyle}=[shape = circle, minimum size = 6pt, inner sep = 1.2pt, draw]
\Vertex[x = 0.55, y = 0.80, L = \small {$2$}]{v0}
\Vertex[x = 0.75, y = 0.80, L = \small {$2$}]{v1}
\Vertex[x = 0.55, y = 0.60, L = \small {$3$}]{v2}
\Vertex[x = 0.75, y = 0.60, L = \small {$3$}]{v3}
\Vertex[x = 0.50, y = 0.45, L = \small {$1$}]{v4}
\Vertex[x = 0.80, y = 0.45, L = \small {$1$}]{v5}
\Vertex[x = 1.05, y = 0.65, L = \small {$2$}]{v6}
\Edge[style = {pre}](v0)(v1)
\Edge[style = {pre}](v2)(v1)
\Edge[style = {post}](v3)(v1)
\Edge[style = {post}](v2)(v0)
\Edge[style = {pre}](v3)(v0)
\Edge[style = {post}](v3)(v2)
\Edge[style = {post}](v0)(v4)
\Edge[style = {pre}](v2)(v4)
\Edge[style = {pre}](v1)(v5)
\Edge[style = {pre}](v3)(v5)
\Edge[style = {pre}](v5)(v4)
\Edge[style = {post}](v1)(v6)
\Edge[style = {pre}](v3)(v6)
\Edge[style = {post}](v5)(v6)
\end{tikzpicture}
\label{fig:A3B3EulerA}
% pic2a
}}
\subfloat[E = 4, O = 2]{
\makebox[.33\textwidth]{
\begin{tikzpicture}[scale = 8]
\tikzstyle{VertexStyle}=[shape = circle, minimum size = 6pt, inner sep = 1.2pt, draw]
\Vertex[x = 0.35, y = 0.80, L = \small {$2$}]{v0}
\Vertex[x = 0.55, y = 0.80, L = \small {$2$}]{v1}
\Vertex[x = 0.30, y = 0.60, L = \small {$1$}]{v2}
\Vertex[x = 0.60, y = 0.60, L = \small {$1$}]{v3}
\Vertex[x = 0.45, y = 0.70, L = \small {$2$}]{v4}
\Edge[style = {post}](v1)(v0)
\Edge[style = {post}](v2)(v0)
\Edge[style = {pre}](v2)(v3)
\Edge[style = {post}](v3)(v1)
\Edge[style = {post}](v0)(v4)
\Edge[style = {pre}](v1)(v4)
\Edge[style = {post}](v2)(v4)
\Edge[style = {pre}](v3)(v4)
\end{tikzpicture}
\label{fig:A3B3EulerB}
% pic2b
}}
%\caption{Subgraphs forbidden by Alon-Tarsi orientations}
%\label{fig:OneSideA3B3}
%\end{figure}
%
\subfloat[EE = 3, EO = 1]{
\makebox[.33\textwidth]{
\begin{tikzpicture}[scale = 8]
\tikzstyle{VertexStyle} = []
\tikzstyle{EdgeStyle} = []
\tikzstyle{labeledStyle}=[shape = circle, minimum size = 6pt, inner sep = 1.2pt, draw]
\tikzstyle{unlabeledStyle}=[shape = circle, minimum size = 6pt, inner sep = 1.2pt, draw]
\tikzstyle{type1}=[post]
\Vertex[style = labeledStyle, x = 0.30, y = 0.45, L = \small {1}]{v0}
\Vertex[style = labeledStyle, x = 0.30, y = 0.65, L = \small {2}]{v1}
\Vertex[style = labeledStyle, x = 0.50, y = 0.65, L = \small {1}]{v2}
\Vertex[style = labeledStyle, x = 0.50, y = 0.45, L = \small {1}]{v3}
\Vertex[style = labeledStyle, x = 0.60, y = 0.75, L = \small {2}]{v4}
\Edge[style = type1, label = \small {}, labelstyle={auto=right, fill=none}](v1)(v0)
\Edge[style = type1, label = \small {}, labelstyle={auto=right, fill=none}](v2)(v1)
\Edge[style = type1, label = \small {}, labelstyle={auto=right, fill=none}](v3)(v2)
\Edge[style = type1, label = \small {}, labelstyle={auto=right, fill=none}](v2)(v4)
\Edge[style = type1, label = \small {}, labelstyle={auto=right, fill=none}](v3)(v4)
\Edge[style = type1, labelstyle={auto=right, fill=none}](v4)(v1)
\Edge[style = type1, label = \small {}, labelstyle={auto=right, fill=none}](v0)(v3)
\label{fig:A3B2Euler}
\end{tikzpicture}}}
%

%\begin{figure}
%\centering
\subfloat[EE = 14, EO = 15]{
\makebox[.33\textwidth]{
\begin{tikzpicture}[scale = 8]
\tikzstyle{VertexStyle}=[shape = circle, minimum size = 6pt, inner sep = 1.2pt, draw]
\Vertex[x = 0.70, y = 0.80, L = \small {$2$}]{v0}
\Vertex[x = 0.50, y = 0.60, L = \small {$2$}]{v1}
\Vertex[x = 0.70, y = 0.60, L = \small {$3$}]{v2}
\Vertex[x = 0.40, y = 0.40, L = \small {$1$}]{v3}
\Vertex[x = 0.80, y = 0.40, L = \small {$2$}]{v4}
\Vertex[x = 0.95, y = 0.65, L = \small {$2$}]{v5}
\Edge[style = {post}](v1)(v0)
\Edge[style = {post}](v2)(v0)
\Edge[style = {post}](v2)(v1)
\Edge[style = {pre}](v1)(v3)
\Edge[style = {post}](v0)(v4)
\Edge[style = {pre}](v2)(v4)
\Edge[style = {post}](v1)(v4)
\Edge[style = {post}](v3)(v2)
\Edge[style = {pre}](v3)(v0)
\Edge[style = {post}](v0)(v5)
\Edge[style = {pre}](v2)(v5)
\Edge[style = {post}](v4)(v5)
\end{tikzpicture}
\label{fig:A3B3EulerC}
% pic3a
}}
\subfloat[EE = 13, EO = 11]{
\makebox[.33\textwidth]{
\begin{tikzpicture}[scale = 8]
\tikzstyle{VertexStyle}=[shape = circle, minimum size = 6pt, inner sep = 1.2pt, draw]
\Vertex[x = 0.60, y = 0.60, L = \small {$2$}]{v0}
\Vertex[x = 0.80, y = 0.60, L = \small {$2$}]{v1}
\Vertex[x = 0.55, y = 0.45, L = \small {$2$}]{v2}
\Vertex[x = 0.85, y = 0.45, L = \small {$1$}]{v3}
\Vertex[x = 0.60, y = 0.80, L = \small {$3$}]{v4}
\Vertex[x = 0.75, y = 0.25, L = \small {$2$}]{v5}
\Edge[style = {post}](v1)(v0)
\Edge[style = {post}](v2)(v0)
\Edge[style = {pre}](v2)(v3)
\Edge[style = {post}](v3)(v1)
\Edge[style = {pre}](v4)(v0)
\Edge[style = {pre}](v4)(v2)
\Edge[style = {post}](v4)(v1)
\Edge[style = {post}](v5)(v2)
\Edge[style = {pre}](v5)(v0)
\Edge[style = {pre}](v5)(v1)
\Edge[style = {post}](v5)(v3)
\Edge[style = {post}](v5)(v4)
\end{tikzpicture}
\label{fig:A3B2Outside}
% pic3b
}}
\subfloat[EE = 5, EO = 3]{
\makebox[.33\textwidth]{
\begin{tikzpicture}[scale = 8]
\tikzstyle{VertexStyle} = []
\tikzstyle{EdgeStyle} = []
\tikzstyle{labeledStyle}=[shape = circle, minimum size = 6pt, inner sep = 1.2pt, draw]
\tikzstyle{type1}=[post]
\Vertex[style = labeledStyle, x = 0.80, y = 0.75, L = \small {1}]{v0}
\Vertex[style = labeledStyle, x = 0.80, y = 0.55, L = \small {2}]{v1}
\Vertex[style = labeledStyle, x = 0.60, y = 0.75, L = \small {1}]{v2}
\Vertex[style = labeledStyle, x = 0.60, y = 0.55, L = \small {2}]{v3}
\Vertex[style = labeledStyle, x = 0.50, y = 0.85, L = \small {2}]{v4}
\Vertex[style = labeledStyle, x = 0.90, y = 0.85, L = \small {5}]{v5}
\Edge[style = type1, label = \small {}, labelstyle={auto=right, fill=none}](v0)(v2)
\Edge[style = type1, label = \small {}, labelstyle={auto=right, fill=none}](v0)(v4)
\Edge[style = type1, label = \small {}, labelstyle={auto=right, fill=none}](v0)(v5)
\Edge[style = type1, label = \small {}, labelstyle={auto=right, fill=none}](v1)(v0)
\Edge[style = type1, label = \small {}, labelstyle={auto=right, fill=none}](v2)(v3)
\Edge[style = type1, label = \small {}, labelstyle={auto=right, fill=none}](v2)(v4)
\Edge[style = type1, label = \small {}, labelstyle={auto=right, fill=none}](v2)(v5)
\Edge[style = type1, label = \small {}, labelstyle={auto=right, fill=none}](v3)(v1)
\Edge[style = type1, label = \small {}, labelstyle={auto=right, fill=none}](v4)(v3)
\Edge[style = type1, label = \small {}, labelstyle={auto=right, fill=none}](v4)(v5)
\Edge[style = type1, label = \small {}, labelstyle={auto=right, fill=none}](v5)(v1)
\label{fig:A2B4Euler}
\end{tikzpicture}}}

\subfloat[\label{E2vC4}: EE = 22, EO = 16]{
\makebox[.33\textwidth]{
\begin{tikzpicture}[scale = 8]
\tikzstyle{VertexStyle}=[shape = circle, minimum size = 6pt, inner sep = 1.2pt, draw]
\Vertex[x = 0.45, y = 0.70, L = \small {$2$}]{v0}
\Vertex[x = 0.45, y = 0.50, L = \small {$2$}]{v1}
\Vertex[x = 0.65, y = 0.75, L = \small {$2$}]{v2}
\Vertex[x = 0.65, y = 0.45, L = \small {$2$}]{v3}
\Vertex[x = 0.85, y = 0.75, L = \small {$2$}]{v4}
\Vertex[x = 0.85, y = 0.45, L = \small {$2$}]{v5}
\Edge[style = {pre}](v0)(v2)
\Edge[style = {post}](v1)(v2)
\Edge[style = {post}](v0)(v3)
\Edge[style = {post}](v1)(v3)
%\Edge[style = {post}](v1)(v0)
\Edge[style = {post}](v5)(v4)
\Edge[style = {pre}](v5)(v3)
\Edge[style = {pre}](v4)(v2)
\Edge[style = {pre}](v2)(v3)
\Edge[style = {pre}](v0)(v4)
\Edge[style = {pre}](v1)(v4)
\Edge[style = {post}](v0)(v5)
\Edge[style = {pre}](v1)(v5)
\end{tikzpicture}}}
\subfloat[\label{fig2g} EE = 72, EO = 74]{
\makebox[.3\textwidth]{
\begin{tikzpicture}[scale = 10]
\tikzstyle{VertexStyle} = []
\tikzstyle{EdgeStyle} = []
\tikzstyle{labeledStyle}=[shape = circle, minimum size = 6pt, inner sep = 1.2pt, draw]
\tikzstyle{type1}=[post]
\Vertex[style = labeledStyle, x = 0.50, y = 0.80, L = \small {2}]{v0}
\Vertex[style = labeledStyle, x = 0.70, y = 0.80, L = \small {2}]{v1}
\Vertex[style = labeledStyle, x = 0.50, y = 0.60, L = \small {2}]{v2}
\Vertex[style = labeledStyle, x = 0.70, y = 0.60, L = \small {2}]{v3}
\Vertex[style = labeledStyle, x = 0.60, y = 0.70, L = \small {2}]{v4}
\Vertex[style = labeledStyle, x = 0.30, y = 0.80, L = \small {3}]{v5}
\Vertex[style = labeledStyle, x = 0.30, y = 0.60, L = \small {4}]{v6}
\Vertex[style = labeledStyle, x = 0.20, y = 0.70, L = \small {2}]{v7}
\Edge[style = type1, label = \small {}, labelstyle={auto=right, fill=none}](v0)(v4)
\Edge[style = type1, label = \tiny {}, labelstyle={auto=right, fill=none}](v0)(v5)
\Edge[style = type1, label = \tiny {}, labelstyle={auto=right, fill=none}](v0)(v6)
\Edge[style = type1, label = \tiny {}, labelstyle={auto=right, fill=none}](v0)(v7)
\Edge[style = type1, label = \tiny {}, labelstyle={auto=right, fill=none}](v1)(v0)
\Edge[style = type1, label = \tiny {}, labelstyle={auto=right, fill=none}](v2)(v0)
\Edge[style = type1, label = \tiny {}, labelstyle={auto=right, fill=none}](v2)(v3)
\Edge[style = type1, label = \tiny {}, labelstyle={auto=right, fill=none}](v2)(v6)
\Edge[style = type1, label = \tiny {}, labelstyle={auto=right, fill=none}](v2)(v7)
\Edge[style = type1, label = \tiny {}, labelstyle={auto=right, fill=none}](v3)(v1)
\Edge[style = type1, label = \tiny {}, labelstyle={auto=right, fill=none}](v4)(v1)
\Edge[style = type1, label = \tiny {}, labelstyle={auto=right, fill=none}](v4)(v2)
\Edge[style = type1, label = \tiny {}, labelstyle={auto=right, fill=none}](v4)(v3)
\Edge[style = type1, label = \tiny {}, labelstyle={auto=right, fill=none}](v4)(v6)
\Edge[style = type1, label = \tiny {}, labelstyle={auto=right, fill=none}](v5)(v2)
\Edge[style = type1, label = \tiny {}, labelstyle={auto=right, fill=none}](v5)(v4)
\Edge[style = type1, label = \tiny {}, labelstyle={auto=right, fill=none}](v6)(v5)
\Edge[style = type1, label = \tiny {}, labelstyle={auto=right, fill=none}](v7)(v5)
\Edge[style = type1, label = \tiny {}, labelstyle={auto=right, fill=none}](v7)(v6)
\label{fig:biggie}
\end{tikzpicture}}}

\caption{Subgraphs forbidden by Alon--Tarsi orientations, used in
Lemma~\ref{NoNonLinear}.
\label{ATpics2}}
\end{figure}
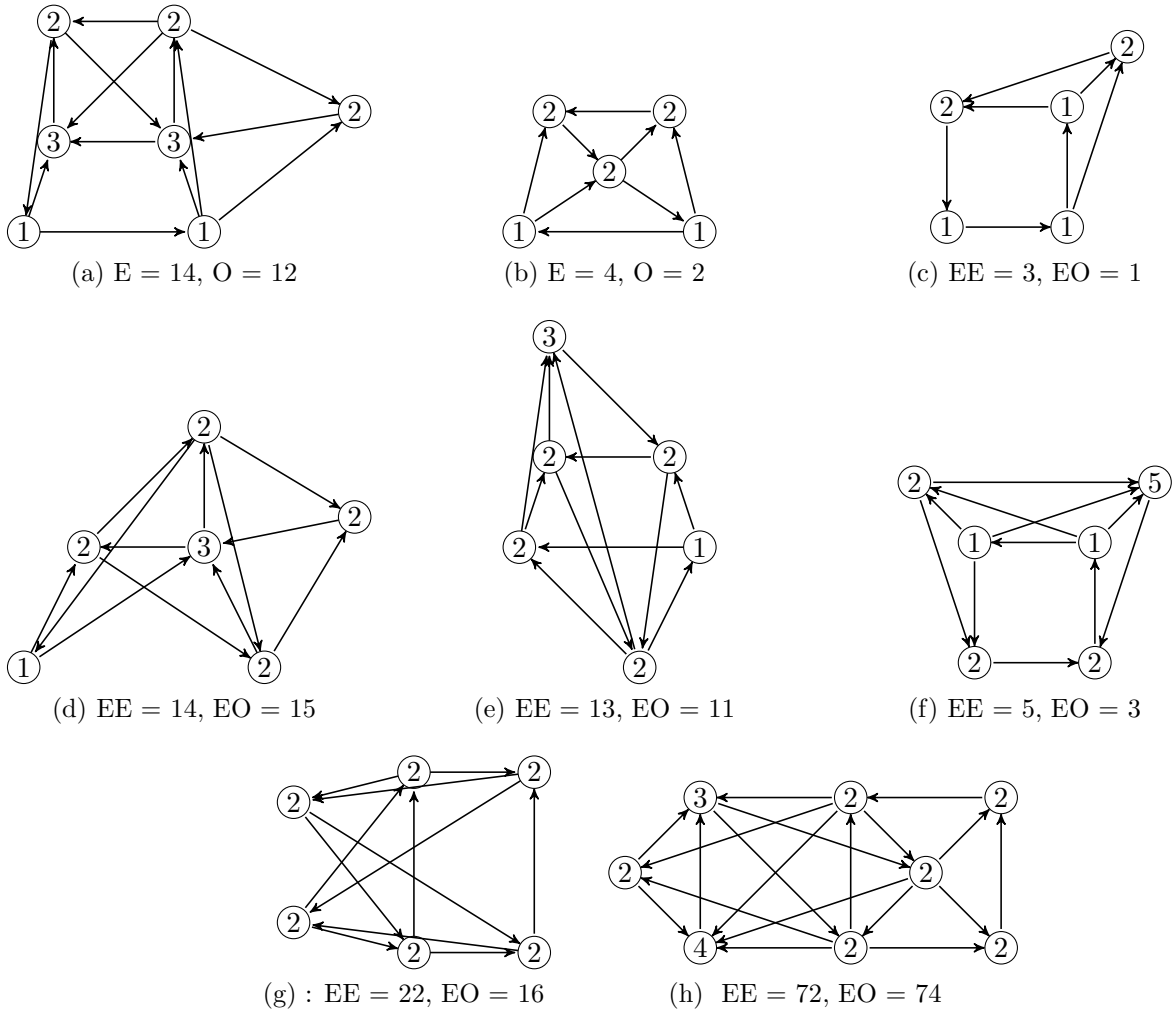

Suppose now that $A\subseteq W$ and $|A|\ge 3$.  Since $\Delta_B-\delta_B\le
1$, each $z_i$ is nonadjacent to exactly one vertex of $A$; call it $u_i$.  By
Claim 0, no vertex in $A$ can have two neighbors outside of $W$; so $t = 1$ and
hence $|B| = 2$.  Now again, $|W|\ge |A|+|B|-1$, so $H$ cannot induce a $C_4$;
this contradicts the hypothesis of the lemma.  Thus, $A\not\subseteq W$.

So assume instead that $|A\cap W| = |A|-1$. 
Note that $|A|\ge 3$, since $|A\cap W|\ge 2$.

Let $\{u\} = A \setminus W$ and $\{w\}= B \cap W$.  Note that $u$ is not
adjacent to $w$, since $u\notin W$.  Since $|W \cap A| \ge 2$, Claim 0 implies
that $\Delta_A \le |W|$.  Hence each $v \in A \cap W$ has no neighbors in $B - w$.
Also $w$ has at least two neighbors in $A$, so $\delta_B \ge |B|$.  
Now $\Delta_B - \delta_B \le 1$ implies that $|A| \le 3$.  Since also $|A|\ge
3$, we have $|A|=3$ and each vertex of $W\setminus B$ is adjacent to $u$.  
Now $|A|=|W| \ge \Delta_A \ge d_H(u) = |A|-1+t$, so $t\le 1$.
Actually $t = 1$, since otherwise $H$ is almost complete, so
it cannot induce $C_4$.  

Now $H$ must be as in Figure~\ref{fig:AandBAre3}\subref{fig:AandBAre3b}, with
$z$ low.  If some vertex outside of $H$ is adjacent to all of $H$, then $G$
contains the $f$-AT subgraph in Figure~\ref{ATpics2}\subref{fig:A3B2Outside}, a
contradiction. 
So each neighbor of $A$ outside of $H$ is adjacent to $A$ and not adjacent to
$B$.  Since $G$ is quasi-line and none of its neighbors outside $H$ is adjacent
to $B$, all of these outside neighbors form a clique.
If some vertex of $A$ is high, then these outside neighbors, together with $A$,
form a $K_{\Delta(G)}$, which is a contradiction.  Otherwise, all of $A$ is low.
In this case, $G$ contains the reducible configuration in
Figure~\ref{ATpics2}\subref{fig:A3B2Euler}.
%\url{http://tinyurl.com/lxv78r8}.
%\end{proof}

By symmetry, we henceforth assume that $|A|\ge |B|$.
%\bigskip

\claim{3}{$A$ is a maximum clique in $H$ and $\Delta_A \le |A|$.}
%
%\begin{proof}
First, suppose $\omega(H) \le 2$.  Recall that $H$ induces a $C_4$, so $|A| \ge 2$.
In fact, we must have $H=C_4$, since $2\ge \omega(H)\ge |A|\ge|B|\ge 2$.
Hence, the degree condition is satisfied.

Now assume $\omega(H) > 2$.  By Claim 2,  no maximum clique in $H$ has vertices
in both $A$ and $B$.  In particular, $A$ is a maximum clique.  
%Now the degree condition follows from Claim~0.
If $\Delta_A > |A|$, then there exists $v\in A$ with at least two neighbors in
$B$.  Since $\Delta_A-\delta_A\le 1$, each other vertex in $A$ has at least one
neighbor in $B$.  Now, since $|A|\ge |B|$, some vertex in $B$ has at least two
neighbors in $A$.  However, now we get $\Delta_A>|A|$ and $\Delta_B>|B|$, which
contradicts Claim~0.  Thus, the degree condition holds.
%\end{proof}

\claim{4}{If all of $A$ is low, then the lemma is true.}
%
%\begin{proof}
First suppose that $|A|=|B|$.  
If all of $B$ is low, then we have an induced $C_4$ in the low
vertex subgraph, which is $f$-AT.  So suppose that some vertex $b\in B$ is
high.  Since $|A|=|B|$, Claim~3 also shows that $\Delta_B\le |B|$.  Hence,
each vertex in $B$ has at most one neighbor in $A$.  So $b$ must have
$\Delta-1$ neighbors in $G\setminus A$.  If all neighbors of $b$ in $G\setminus
H$ induce a clique, then $G$ contains a copy of $K_{\Delta(G)}$, which
contradicts that $\omega(G)<\Delta(G)$.  So $b$ has
nonadjacent neighbors $u_1,u_2$ in $G\setminus H$.  Since $G$ is quasi-line, at
least one of $u_1$ and $u_2$ is complete to $A$; by symmetry, say this is
$u_1$.  Now consider an induced $C_4$ in $H$, together with $u_1$.
Since all of $A$ is low, this is the configuration shown in
Figure~\ref{ATpics2}\subref{fig:A3B3EulerB}, which is $f$-AT.

So assume instead that $|A|>|B|$.  Since $\Delta_A\le |A|$ (and all of $A$ is
low), each vertex of $A$ must have exactly one neighbor in $B$.  Since
$\Delta_B-\delta_B\le 1$, for some integer $k$, each vertex in $B$ has either
$k$ or $k+1$ neighbors in $A$.  Since $|A|>|B|$, we have $k\ge 1$.
If there exist $b_1,b_2\in B$ each with at least two neighbors in $A$, then we
have the configuration in Figure~\ref{ATpics2}\subref{fig:A2B4Euler},
%\url{http://tinyurl.com/q6zx2ku}, 
which is $f$-AT.  So we may assume that
$k=1$ and $B$ has at most one high vertex.  If $B$ has at least two low
vertices, then we have an induced $C_4$ of low vertices, which is $f$-AT, a
contradiction.  So $B$ must contain exactly one high vertex and one low vertex.
Now we have the configuration in Figure~\ref{ATpics2}\subref{fig:A3B2Euler},
%have \url{http://tinyurl.com/o6x42t5}, 
which is $f$-AT.  
%
%So assume instead that $|A|>|B|$.  Since $\Delta_A\le |A|$ (and all of $A$ is
%low), each vertex of $A$ must have exactly one neighbor in $B$.  Since
%$\Delta_B-\delta_B\le 1$, for some integer $k$, each vertex in $B$ has either
%$k$ or $k+1$ neighbors in $A$.  Since $|A|>|B|$, we have $k\ge 1$.
%Suppose that $k>1$.  Let $b$ be a vertex of $B$ and let $a_1,a_2,a_3$ be
%neighbors of $B$ in $A$.  Since $|B|\ge 2$ (and each vertex of $A$ has only one
%neighbor in $B$), some vertex $a_4\in A$ is not adjacent to $b$.  Let $Q$ be
%the subgraph induced by $a_1,\ldots,a_4,b$. Then $Q$ is $K_5-e$, with all
%vertices low except for $b$ and hence is $f$-AT (as shown in the painting
%squares paper), a contradiction. 
%Hence, $k=1$.  
%
%Now consider the number of low vertices in $B$.  If this is at
%least 2, then we have an induced $C_4$ of low vertices, which is $f$-AT, a
%contradiction.  If it is 1, then we have the $f$-AT configuration in
%Figure~\ref{NonLinear-fig}(a).
%%\url{http://tinyurl.com/o6x42t5},
%If it is 0, then we have the $f$-AT configuration in %\url{http://tinyurl.com/q6zx2ku}
%Figure~\ref{NonLinear-fig}(b).
%%which is $f$-AT. 
%\end{proof}

\claim{5}{There exists a unique vertex $w$ that is joined to all of $H$.}
%
%The set $A\cup N(A)\setminus(B\cup\set{w})$ induces $K_{\Delta-1}$; similarly,
%the set $B\cup N(B)\setminus(A\cup\set{w})$ induces $K_{\Delta-1}$.}
%\begin{proof}
Since $\Delta_A \le |A|$, each vertex of $A$ has at most one neighbor in $B$.
Since $A$ is not all low, $A\cup N(A)\setminus H$ 
%together with its neighborhood outside $H$ 
has $\Delta$ vertices.  Since $G$ does not contain
$K_{\Delta}$, some pair of neighbors of $A$ in $G\setminus H$ must be nonadjacent.
Since $G$ is quasi-line, one of those neighbors is joined to $H$; call this vertex $w$. 
If two vertices outside $H$ are joined to $H$, then $G$ contains 
the $f$-AT configuration in Figure~\ref{ATpics1}\subref{K2vC4} or
Figure~\ref{ATpics2}\subref{E2vC4}.  Thus, $w$ is unique.
%Thus, $w$ is unique.  
%This contradiction proves the claim.
%
%But now $G$ has an induced copy of the $f$-AT configuration in
%Figure~\ref{NonLinear-fig}(c).
%%\url{http://tinyurl.com/mlz3ze6} and 
%Hence $G$ is not BK-free, a contradiction.
%\end{proof}
%\RED{Why is $|A|\ge 4$?}
%
\bigskip

If $|A|\ge 4$, then $G$ contains the $f$-AT subgraph in
Figure~\ref{ATpics2}\subref{fig:biggie}.  So assume $\card{A}\le 3$.
Suppose $\Delta_B>\card{B}$.  Since $3\ge \card{A}\ge\card{B}\ge 2$ (and
$\Delta_A\le\card{A}$), we have $\card{A}=3$ and $\card{B}=2$.  Now $G$ contains
the $f$-AT subgraph in Figure~\ref{ATpics2}\subref{fig:A3B2Outside} or
Figure~\ref{ATpics2}\subref{fig:A3B3EulerB} (if the two vertices in $B$ have a
common neighbor in $A$).  So we conclude that $\Delta_B\le\card{B}$.

If all of $B$ is low, then $G$ contains the $f$-AT subgraph in
Figure~\ref{ATpics2}\subref{fig:A3B3EulerB}.  So instead $B$ contains some high
vertex $b$.  Since $\Delta_B\le \card{B}$, and $b$ is high, $B\cup N(B)\setminus
H$ contains $\Delta$ vertices.  
%Now the argument above for $A$ applies to $B$.  
Since $G$ contains no $K_\Delta$, %and since $N(b)$ is covered by two cliques, 
the set $N(B)\setminus H$ contains some nonneighbor of $w$.
If $\card{B\cup (N(B)\setminus (H\cup\{w\}))}\ge 4$, then $G$ contains the
$f$-AT subgraph in Figure~\ref{ATpics2}\subref{fig:biggie}.  The same is true
if $\card{A\cup (N(A)\setminus (H\cup\{w\}))}\ge 4$.  
%Note that no neighbor of $w$ is joined to $H$.  Since $G$ is quasi-line,
%each neighbor $v$ of $w$ outside $H$ is in $N(H)$, so $v$ is in exactly one of
%$N(A)\setminus H$ and $N(B)\setminus H$.
Since $G$ is quasi-line, $N(w)$ is contained in $H\cup
N(H)=A\cup B\cup N(A)\cup N(B)$.  This gives $d(w)\le 3+3$, which contradicts
that $\delta(G)\ge \Delta(G)-1\ge 8$.
This contradiction finishes the proof of the lemma.
%More generally, this is true if $w$
%has at least 4 neighbors in $A\cup N(A)\setminus H$.
%
%Now $G$ contains the $f$-AT subgraph in Figure~\ref{ATpics2}\subref{fig:biggie}.
\end{proof}

\subsection{Handling 2-joins}
\label{2joins}
Our goal in this section is to write $G$ as a composition of linear
interval strips, where each strip is complete or complete less an edge (since
this implies that $G$ is very nearly a line graph).  Our main tool is the
following lemma, which we will apply to each interval
2-join in the representation. 

\begin{lem}\label{Irreducible2Join}
Let $G$ be a BK-free graph with $\Delta(G) \ge 9$ and $\omega(G)<\Delta(G)$.
If $(H, A_1, A_2, B_1, B_2)$ is an irreducible canonical interval $2$-join in $G$, then 
\begin{enumerate}
\item[(1)] $B_1 \cap B_2 = \emptyset$; and,
\item[(2)] $\card{A_1}, \card{A_2} \le 3$; and,
\item[(3)] either $H$ is complete, or $H = K_{|H|} - xy$ and $|H|\le 6$, where
$x$ and $y$ are low in $G$.
\end{enumerate}
\end{lem}
\begin{proof}
The most interesting of the three conclusions in the lemma is (3).
If $H$ is complete, for every choice of $H$, then $G$ is a line graph, which we
handle in Section~\ref{line-graphs}.  So (3) proves that $G$ is quite close to
being a line graph.

Let $(H, A_1, A_2, B_1, B_2)$ be an irreducible canonical interval $2$-join in $G$. 
%Let $\Delta \DefinedAs \Delta(G)$.  
Note that $G$ has no simplicial vertices, since $\delta(G)\ge \Delta(G)-1$ and
$\omega(G)<\Delta(G)$.
Label the vertices of $H$ left-to-right as $v_1, \ldots, v_t$.  Say $A_1 = \set{v_1, \ldots, v_L}$ and $A_2 = \set{v_R, \ldots, v_t}$. 
For $v \in V(H)$, let $r(v) \DefinedAs \max\setbs{i \in \irange{t}}{v \adj
v_i}$ and $l(v) \DefinedAs \min\setbs{i \in \irange{t}}{v \adj v_i}$.  These
are well-defined since $\card{H} \geq 2$ and $H$ is connected by the following claim.

\claim{0} {$H$ is connected and each of $A_1, A_2, B_1, B_2$ is nonempty.}
Otherwise $G$ contains a simplicial vertex.

\claim{1} {If $H$ is incomplete, then $r(v_L) = r(v_1) + 1$ and $l(v_R) =
l(v_t) - 1$.  In particular, $v_1$ and $v_t$ are low and also $\card{A_1}\ge 2$ and
$\card{A_2} \geq 2$.} Suppose instead that $H$ is incomplete and $r(v_L) \neq
r(v_1) + 1$.  By definition, $N_H(v_1) \subseteq N_H(v_L)$ and $v_1$ and $v_L$ have
the same neighbors
in $G\setminus H$.  If $r(v_L) = r(v_1)$, then $N_H(A_1)\setminus A_1 =
N_H(v_1)\setminus A_1$, so $H$ is reducible, which is a contradiction.  Thus
$r(v_L) \ge r(v_1)+1$.  If $r(v_L) \ge r(v_1) + 2$, then $d(v_L) - d(v_1) \geq
2$, which is impossible, since $\delta(G)\ge \Delta(G)-1$.  
%So we must have $r(v_L) = r(v_1)$ and hence $N_H(A_1)\setminus
%A_1 = N_H(v_1)\setminus A_1$.  Thus the $2$-join is reducible, a contradiction.
So $r(v_L) = r(v_1) + 1$, as desired.  Similarly, $l(v_R) = l(v_t) - 1$. 

\claim{2} {If $H$ is complete or complete less an edge, then $R - L = 1$.} 
Assume, for a contradiction, that $R-L\ne 1$, so $V(H) \neq A_1 \cup A_2$. 
%Since $H$ is canonical, $A_1\cap A_2=\emptyset$.  
First suppose that
$H$ is complete.   Now any $v \in V(H) \setminus A_1 \cup A_2$ is simplicial in
$G$, which is a contradiction.  So suppose instead that $H$ is complete less an
edge, and choose $v \in V(H) \setminus(A_1 \cup A_2)$.  Now $N[v]$ is
complete less an edge; since $G$ has no $K_\Delta$, $v$ must
be low.  By Claim 1, $v_1$ and $v_t$ are also low, so $G$ contains a copy of
$K_4-e$ in which one vertex in both triangles is high and the other three
vertices are low.  This subgraph is $f$-AT, as shown in
Figure~\ref{ATpics1}\subref{K4-e}, which is a contradiction.

\claim{3} {$B_1 \not \subseteq B_2$, $B_2 \not \subseteq B_1$.}
If not, then by symmetry we can assume $B_2 \subseteq B_1$.  First, suppose $H$
is complete or complete less an edge.  By Claim 2, $R-L=1$.
If $H$ is complete, then the vertices in $A_2$ are simplicial, which is
impossible.  If $H$ is complete less an edge, then for a high vertex $v$ in
$A_2$ (which exists by Claim~1), $N[v]$ induces
$K_{\Delta+1} - e$; this contains $K_\Delta$, which is a contradiction.

So $H$ is neither complete nor complete less an edge; in particular $v_1\nonadj
v_t$.  If $v_1\adj v_{t-1}$, then $v_{t-1}$ is high, since $d(v_t)<d(v_{t-1})$.
 This implies $v_t\adj v_2$; now $H$ is complete less an edge, which is a
contradiction.  So $v_1\nonadj v_{t-1}$ and, by symmetry, $v_2\nonadj v_t$.
%Now $v_1 \nonadj v_{t-1}, v_t$ and $v_t \nonadj v_2$.  
%
If $|B_2| \ge 2$, then since $|A_1|\ge 2$ and $|A_2| \ge 2$ by Claim~1, 
then consider the subgraph induced by %two vertices from each of
$v_1,v_L,V_R,v_t$, and two vertices of $B_2$.  Since $B_2\subseteq B_1$, this
induced subgraph is either Figure~\ref{2JoinPics}\subref{fig:K22K2low} or
Figure~\ref{2JoinPics}\subref{fig:K2P4low}, which is
a contradiction, since $G$ is $BK$-free.

So we must have $|B_2|=1$.  Let $\set{w}=B_2$.
Now $v_t$ is in a $K_{\Delta-1}$ in $H$, say with vertices $v_q, v_{q+1},
\ldots, v_t$.  In particular, $w$ is not joined to $H$, so $R-L \ne 1$.  If
$|A_2| \ge 4$, then $\{v_t,v_{t-1},v_{t-2},v_{t-3},v_q,w\}$ induces
$\join{K_4}{E_2}$, where the $K_4$ has a low vertex, $v_t$.  As shown in
Figure~\ref{ATpics1}\subref{K4vE2}, this is $f$-AT, which is a contradiction.
So $|A_2| \le 3$.  

First, suppose $v_{R-1}$ is low. Now $l(v_{R-1}) = q - 1$.
%if $v_{R-1}$ is high, then $l(v_{R-1}) = q - 2$.  
Since $|A_2| \le 3$, the subgraph induced by $\{v_t, v_{R-1}, v_{R-2}, v_{R-3},
v_{R-4}, v_{q-1}\}$ is $\join{K_4}{E_2}$, with a low vertex in the $E_2$.
This is $f$-AT by Figure~\ref{ATpics1}\subref{K4vE2}, which is a contradiction.
So assume instead that $v_R$ is high.  Now $l(v_{R-1})=q-2$, so 
the subgraph induced by
$\{v_t,v_{t-1},v_{R-1},v_{R-2},v_{R-3},v_{R-4},v_{q-1},v_{q-2}\}$ is
$\join{K_4}{B}$, where $B$ is not almost complete.  This subgraph is $f$-AT, as
shown in Figures~\ref{ATpics1}\subref{K4v2E2c}--\ref{ATpics1}\subref{K2vC4},
which is a contradiction.
%we have a
%$K_{\Delta - 4}$ joined to $A_2$ and to $\set{v_{q-2}, v_{q-1}}$ if $v_{R-1}$
%is high and to $\set{v_{q-1}}$ if $v_{R-1}$ is low.  So, when $v_{R-1}$ is
%high, we have $K_4$ joined a graph that is not almost complete, this is
%reducible.  When $v_{R-1}$ is low, we have $\join{K_4}{E_2}$, where the $K_4$
%has a low vertex, which is also reducible.

%So, we must have $|B_2| = 1$. Then each vertex in $A_2$ is in a $K_{\Delta-1}$ in $H$ and the $K_{\Delta-1}$s of $v_t, v_{t-1}$ together give a $K_\Delta$ less an edge in $H$.  Just like in the circular interval graphs proof, look at a vertex in the middle of this and see how its neighbors can hang of the ends (only the left end is possible in this case).  If the vertex is high, we get $\join{K_4}{P_4}$ if it is low, we get $\join{K_4}{P_3}$ where a vertex in the $K_4$ is low. Both are impossible in a $BK$-critical graph.

\claim{4} {$\card{A_1}, \card{A_2} \leq 3$.}  Suppose otherwise, by symmetry, that
$\card{A_1} \geq 4$.  First, suppose $H$ is complete.  By Claim~2, $V(H) = A_1
\cup A_2$. If $v_1$ is low, then for any $w_1 \in B_1 \setminus B_2$ the vertex
set $\{v_1, \ldots, v_4, v_t, w_1\}$ induces a $\join{K_4}{E_2}$, which 
contradicts Figures~\ref{ATpics1}\subref{K4vE2}.
Hence $v_1$ is high. If $\card{A_2} \geq 2$ and $\card{B_1 \setminus B_2} \geq 2$, 
then for any $w_1, w_2 \in B_1 \setminus B_2$, the vertex set $\{v_1, \ldots,
v_4, v_{t-1}, v_t, w_1, w_2\}$ induces $\join{K_4}{2K_2}$, which contradicts
Figure~\ref{ATpics1}\subref{K4v2E2c}. 
Hence either $\card{A_2} = 1$ or $\card{B_1 \setminus B_2} = 1$.  Suppose
$\card{A_2} = 1$.  Since $A_1 \cup B_1$ induces a clique and $\card{A_1 \cup
B_1} = d(v_1)$, $v_1$ must be low, which is impossible.    Hence, we have
$\card{B_1 \setminus B_2} = 1$, so $\card{B_1 \cap B_2} = \card{B_1} - 1$.
Hence, $V(H) \cup (B_1 \cap B_2)$ induces a clique of size $\card{A_1} + \card{A_2} +
\card{B_1} - 1 = d(v_1) = \Delta$, which is a contradiction.

So $H$ must be incomplete.  By Claim~1, $v_1$ is low.  Now, as above,
for any $w_1 \in B_1 \setminus B_2$, the vertex set $\{v_1, \ldots, v_4,
v_{L+1}, w_1\}$ induces a $\join{K_4}{E_2}$ that contradicts
Figure~\ref{ATpics1}\subref{K4vE2}.  Hence, $\card{A_1} \leq 3$ and,
by symmetry, $\card{A_2} \leq 3$.
% Similarly, 

\claim{5} {$R - L = 1$.}  Suppose otherwise that $R - L \geq 2$.  By Claim~2,
$H$ is incomplete.  Now by Claim~1, $r(v_L) = r(v_1) + 1$, $l(v_R) = l(v_t) -
1$, $v_1$ and $v_t$ are low, and $\card{A_1}\ge 2$ and $\card{A_2} \geq 2$.
Now we will find an $f$-AT subgraph induced by some vertices of $H$.  To this
end, we describe $N(v_{L+1}),N(v_{L+2}),N(v_{L+3}),N(v_{L+4})$.

\subclaim{5a} {$L + \Delta - 2 \leq r(v_{L+1}) \leq L + \Delta - 1$.}
Since $v_{L+1}$ has exactly $L$ neighbors to the left, $r(v_{L+1}) \leq
L + 1 + \Delta - L = \Delta + 1 \leq L + \Delta - 1$.  If $v_{L+1}$ is high,
this computation is exact, so $r(v_{L+1}) = \Delta + 1 \geq L + \Delta - 2$. 
So suppose instead that $v_{L+1}$ is low. If $L=3$, then for some $w_1 \in B_1$
the vertex set $\set{v_1, v_2, v_3, v_4, w_1}$ induces a $\join{K_3}{E_2}$
that contradicts Figure\ref{ATpics1}\subref{K3vE2}.  Hence $L=2$ and
$r(v_{L+1}) = L + 1 + \Delta - 1 - L = \Delta \geq L + \Delta - 2$.

\subclaim{5b} {$L + \Delta - 2 \leq r(v_{L+2}) \leq L + \Delta$.}  By
Subclaim~5a, $r(v_{L+2}) \geq L + \Delta - 2$.  Since $H$ contains no
$\Delta$-clique, $v_{L+2}$ has at least $2$ neighbors to the left if it is high
and at least $1$ neighbor to the left if it is low.  Thus $r(v_{L+2}) \leq L + 2
+ \Delta - 2 = L + \Delta$.

\subclaim{5c} {If $v_{L+4}$ is high, then $l(v_{L+4}) \leq L$.} 
Suppose otherwise.  Recall that $v_{L+1}\adj v_{L+4}$, since $d(v_{L+1})\ge
\Delta-1\ge 8$ and $|A_1|\le 3$.  Now $v_{L+4}$ has exactly 3 neighbors to the
left, so $r(v_{L+4}) = L + \Delta + 1$.  Consider the subgraph induced on
$\{v_{L+1}, v_{L+2}, v_{L+4}, v_{L+5}, v_{L+6}, v_{L+7}$, $v_{L+9}$,
$v_{L+10}\}$.  By Subclaims~5a and~5b, this subgraph 
contradicts Figure~\ref{ATpics1}\subref{K4v2E2c} or
Figure~\ref{ATpics1}\subref{K3vP4}.

\subclaim{5d} {$l(v_{L+3}) \leq L$.} 
Suppose otherwise.  Since $v_{L+1}\adj v_{L+3}$, vertex
$v_{L+3}$ has exactly $2$ neighbors to the left, so
$r(v_{L+3}) \geq L + \Delta$.  By Subclaim~5c, $v_{L+4}$ is low. By Subclaim~5a,
$L + \Delta - 2 \leq r(v_{L+1}) \leq L + \Delta - 1$. Therefore $\set{v_{L+1},
v_{L+3}, v_{L+4}, v_{L+5}, v_{L+6}, v_{L+\Delta}}$ induces a $\join{K_4}{E_2}$
that contradicts Figure~\ref{ATpics1}\subref{K4vE2}.

\subclaim{5e} {$r(v_1) \geq L + 2$.} 
By Subclaim~5d, $r(v_L) \geq L + 3$, so Claim~1 implies that $r(v_1) \geq L + 2$.

\subclaim{5f} {Claim~5 is true.} 
If $r(v_{r(v_1)-1})=r(v_1)+1$, then $v_{r(v_1)-1}$ is low, so
$\{v_1,v_{L+1},v_{L+2},v_{L+3},v_{r(v_1)-1},v_{r(v_1)+1}\}$ 
induces $\join{K_4}{E_2}$, %either Figure~\ref{ATpics1}\subref{K3vE2} or
which contradicts Figure~\ref{ATpics1}\subref{K4vE2}.  So assume 
$r(v_{r(v_1)-1})\ne r(v_1)+1$. 

Consider the subgraph $Q$ induced on $\{v_1, v_L, v_{r(v_1) - 1}, v_{r(v_1)}$,
$v_{r(v_1) + 1}, v_{r(v_{r(v_1) - 1})}\}$; these vertices must be distinct.  
Both $v_{r(v_1) - 1}$ and $v_{r(v_1)}$ are dominating vertices in $Q$. We show
that $\set{v_1,v_L,v_{r(v_1) + 1}, v_{r(v_{r(v_1) - 1})}}$ induces a $P_4$,
so $Q$ is Figure \ref{ATpics1}\subref{fig:K2P4low}, which is a contradiction. 
By definition, $v_1 \adj v_L$, $v_1 \nonadj v_{r(v_1) + 1}$, and $v_1\nonadj
v_{r(v_{r(v_1) - 1})}$. By Claim~2, $v_L \adj v_{r(v_1) + 1}$. 
By Subclaim~5e, $r(v_1) \ge L + 2$, so $r(v_1) - 1 \ge L+1$.  Since $|B_1| > 0$
by Claim 0, this means $r(v_{r(v_1) - 1}) - (r(v_1) - 1) \ge r(v_L) - L$ and
hence $r(v_{r(v_1) - 1}) \ge r(v_L) - L + (r(v_1) - 1) \ge r(v_L) + 1$. 
Therefore $v_L \nonadj v_{r(v_{r(v_1) - 1})}$, so $\set{v_1,v_L,v_{r(v_1) + 1},
v_{r(v_{r(v_1) - 1})}}$ induces a $P_4$ as desired.

\claim{6} {$B_1 \cap B_2 = \emptyset$.} Suppose otherwise that we have
$w \in B_1 \cap B_2$. 

\subclaim{6a} {Each $v \in V(H)$ is low, $\card{B_1} = \card{B_2}$,
$\card{B_1 \setminus B_2} = \card{B_2 \setminus B_1} = 1$, $d(v) = \card{A_1} +
\card{A_2} + \card{B_1} - 1$ for each $v \in V(H)$ and $H$ is complete.} By
Claim~5, we have $d(v) \leq \card{A_1} + \card{A_2} + \card{B_1} - 1$ for each
$v \in A_1$ and $d(v) \leq \card{A_1} + \card{A_2} + \card{B_2} - 1$ for each
$v \in A_2$.  Since $B_1 \not \subseteq B_2$ and $B_2 \not \subseteq B_1$, we
have $d(w) \geq \max\set{\card{B_1}, \card{B_2}} +
\card{A_1} + \card{A_2}$.  So $d(w) \geq d(v) + 1$ for every $v \in V(H)$.  This
implies that each $v \in V(H)$ is low, $\card{B_1} = \card{B_2}$, $\card{B_1
\setminus B_2} = \card{B_2 \setminus B_1} = 1$, and $d(v) = \card{A_1} +
\card{A_2} + \card{B_1} - 1$ for each $v \in V(H)$. Hence, $H$ is complete.  

\subclaim{6b} {$\card{B_1 \cap B_2} \leq 3$.} Suppose otherwise that
$\card{B_1 \cap B_2} \geq 4$.  Pick $w_1 \in B_1 \setminus B_2$, $w_2 \in B_2
\setminus B_1$ and $z_1, z_2, z_3, z_4 \in B_1 \cap B_2$.  Since $v_1\nonadj
w_2$ and $v_t\nonadj w_1$, the set $\set{z_1, z_2, z_3, z_4, w_1, w_2, v_1,
v_t}$ induces an $f$-AT subgraph shown in Figure~\ref{ATpics1}\subref{K4v2E2c},
\ref{ATpics1}\subref{K3vP4}, or Figure~\ref{ATpics1}\subref{K2vC4}; each case
yields a contradiction.  Hence $\card{B_1 \cap B_2} \leq 3$.

\subclaim{6c} {Claim~6 is true.}  By Subclaim~6a and Subclaim~6b we
have $3 \geq \card{B_1 \cap B_2} = \card{B_1} - 1$, so $\card{B_1} =
\card{B_2} \leq 4$.  If $\card{A_1} \leq 2$ and $\card{A_2} \leq 2$, then $\Delta - 1
= d(v_1) \leq 3 + \card{B_1} \leq 7$, which is a contradiction.  Hence, by
symmetry, we assume that $\card{A_1} \geq 3$.  But now for any $w_1 \in B_1
\setminus B_2$, the set $\set{v_1, v_2, v_3, v_t, w_1}$ induces a
$\join{K_3}{E_2}$ contradicting Figure~\ref{ATpics1}\subref{K3vE2}.

\claim{7} {Either $H$ is complete, or $H = K_{|H|} - xy$ where $x$ and $y$ are
low in $G$ and $|H|\le 6$.}  Suppose $H$ is incomplete.  By Claim~5, $R - L =
1$. So, by Claim~1 $r(v_L) = r(v_1) + 1$ and $l(v_R) = l(v_t) - 1$.  Since
$v_1$ is not simplicial, $r(v_1) \geq L + 1 = R$.  Hence $l(v_R) = 1$, so
$l(v_t) = 2$.  Similarly, $r(v_1) = t - 1$.  So, $H$ is $K_t$ less an edge and
$v_1$ and $v_t$ are low (by Claim~2).  Finally, by Claims 5 and 4,
$|H|=|A_1|+|A_2|\le 3+3=6$.
\end{proof}

\begin{figure}[ht]
\subfloat[][\label{fig:K22K2low}~~EE = 8, EO = 9\\$\join{K_2}{2K_2}$ with
two lows]{\makebox[.50\textwidth]{
\begin{tikzpicture}[scale = 10]
\tikzstyle{VertexStyle} = []
\tikzstyle{EdgeStyle} = []
\tikzstyle{labeledStyle}=[shape = circle, minimum size = 6pt, inner sep = 1.2pt, draw]
\tikzstyle{unlabeledStyle}=[shape = circle, minimum size = 6pt, inner sep = 1.2pt, draw, fill]
\tikzstyle{72aaff53-67b1-533a-c226-a2038d6dbcc1}=[post]
\tikzstyle{c4c35e5d-2ab6-3243-0be6-5bbe0efaf8b4}=[pre]
\Vertex[style = labeledStyle, x = 0.500, y = 0.650, L = \small {$2$}]{v0}
\Vertex[style = labeledStyle, x = 0.600, y = 0.650, L = \small {$3$}]{v1}
\Vertex[style = labeledStyle, x = 0.400, y = 0.800, L = \small {$2$}]{v2}
\Vertex[style = labeledStyle, x = 0.400, y = 0.900, L = \small {$1$}]{v3}
\Vertex[style = labeledStyle, x = 0.700, y = 0.800, L = \small {$2$}]{v4}
\Vertex[style = labeledStyle, x = 0.700, y = 0.900, L = \small {$1$}]{v5}
\Edge[style = 72aaff53-67b1-533a-c226-a2038d6dbcc1, label = \tiny {}, labelstyle={auto=right, fill=none}](v2)(v0)
\Edge[style = c4c35e5d-2ab6-3243-0be6-5bbe0efaf8b4, label = \tiny {}, labelstyle={auto=right, fill=none}](v3)(v0)
\Edge[style = c4c35e5d-2ab6-3243-0be6-5bbe0efaf8b4, label = \tiny {}, labelstyle={auto=right, fill=none}](v2)(v1)
\Edge[style = 72aaff53-67b1-533a-c226-a2038d6dbcc1, label = \tiny {}, labelstyle={auto=right, fill=none}](v3)(v1)
\Edge[style = 72aaff53-67b1-533a-c226-a2038d6dbcc1, label = \tiny {}, labelstyle={auto=right, fill=none}](v0)(v4)
\Edge[style = c4c35e5d-2ab6-3243-0be6-5bbe0efaf8b4, label = \tiny {}, labelstyle={auto=right, fill=none}](v1)(v4)
\Edge[style = 72aaff53-67b1-533a-c226-a2038d6dbcc1, label = \tiny {}, labelstyle={auto=right, fill=none}](v0)(v5)
\Edge[style = c4c35e5d-2ab6-3243-0be6-5bbe0efaf8b4, label = \tiny {}, labelstyle={auto=right, fill=none}](v1)(v5)
\Edge[style = c4c35e5d-2ab6-3243-0be6-5bbe0efaf8b4, label = \tiny {}, labelstyle={auto=right, fill=none}](v4)(v5)
\Edge[style = 72aaff53-67b1-533a-c226-a2038d6dbcc1, label = \tiny {}, labelstyle={auto=right, fill=none}](v1)(v0)
\Edge[style = 72aaff53-67b1-533a-c226-a2038d6dbcc1, label = \tiny {}, labelstyle={auto=right, fill=none}](v3)(v2)
\end{tikzpicture}}}
\subfloat[][\label{fig:K2P4low}~~EE =  14, EO = 15\\$\join{K_2}{P_4}$ with one low end]{
\makebox[.50\textwidth]{
\begin{tikzpicture}[scale = 10]
\tikzstyle{VertexStyle} = []
\tikzstyle{EdgeStyle} = []
\tikzstyle{labeledStyle}=[shape = circle, minimum size = 6pt, inner sep = 1.2pt, draw]
\tikzstyle{unlabeledStyle}=[shape = circle, minimum size = 6pt, inner sep = 1.2pt, draw, fill]
\tikzstyle{9387c6a6-9fa2-ac3b-0480-d4ec634f10ba}=[post]
\tikzstyle{f911ea52-03eb-3466-3236-2637d51acb1c}=[pre]
\Vertex[style = labeledStyle, x = 0.500, y = 0.650, L = \small {$2$}]{v0}
\Vertex[style = labeledStyle, x = 0.600, y = 0.650, L = \small {$3$}]{v1}
\Vertex[style = labeledStyle, x = 0.400, y = 0.800, L = \small {$2$}]{v2}
\Vertex[style = labeledStyle, x = 0.400, y = 0.900, L = \small {$1$}]{v3}
\Vertex[style = labeledStyle, x = 0.700, y = 0.800, L = \small {$2$}]{v4}
\Vertex[style = labeledStyle, x = 0.700, y = 0.900, L = \small {$2$}]{v5}
\Edge[style = 9387c6a6-9fa2-ac3b-0480-d4ec634f10ba, label = \tiny {}, labelstyle={auto=right, fill=none}](v0)(v3)
\Edge[style = 9387c6a6-9fa2-ac3b-0480-d4ec634f10ba, label = \tiny {}, labelstyle={auto=right, fill=none}](v0)(v4)
\Edge[style = 9387c6a6-9fa2-ac3b-0480-d4ec634f10ba, label = \tiny {}, labelstyle={auto=right, fill=none}](v0)(v5)
\Edge[style = 9387c6a6-9fa2-ac3b-0480-d4ec634f10ba, label = \tiny {}, labelstyle={auto=right, fill=none}](v1)(v0)
\Edge[style = f911ea52-03eb-3466-3236-2637d51acb1c, label = \tiny {}, labelstyle={auto=right, fill=none}](v1)(v2)
\Edge[style = 9387c6a6-9fa2-ac3b-0480-d4ec634f10ba, label = \tiny {}, labelstyle={auto=right, fill=none}](v2)(v0)
\Edge[style = 9387c6a6-9fa2-ac3b-0480-d4ec634f10ba, label = \tiny {}, labelstyle={auto=right, fill=none}](v3)(v1)
\Edge[style = 9387c6a6-9fa2-ac3b-0480-d4ec634f10ba, label = \tiny {}, labelstyle={auto=right, fill=none}](v3)(v2)
\Edge[style = 9387c6a6-9fa2-ac3b-0480-d4ec634f10ba, label = \tiny {}, labelstyle={auto=right, fill=none}](v4)(v1)
\Edge[style = f911ea52-03eb-3466-3236-2637d51acb1c, label = \tiny {}, labelstyle={auto=right, fill=none}](v5)(v1)
\Edge[style = 9387c6a6-9fa2-ac3b-0480-d4ec634f10ba, label = \tiny {}, labelstyle={auto=right, fill=none}](v5)(v4)
\Edge[style = 9387c6a6-9fa2-ac3b-0480-d4ec634f10ba, label = \tiny {}, labelstyle={auto=right, fill=none}](v4)(v2)
\end{tikzpicture}}}
\caption{Subgraphs forbidden by Alon--Tarsi orientations, used in
Lemma~\ref{Irreducible2Join}.\label{2JoinPics}}
\end{figure}
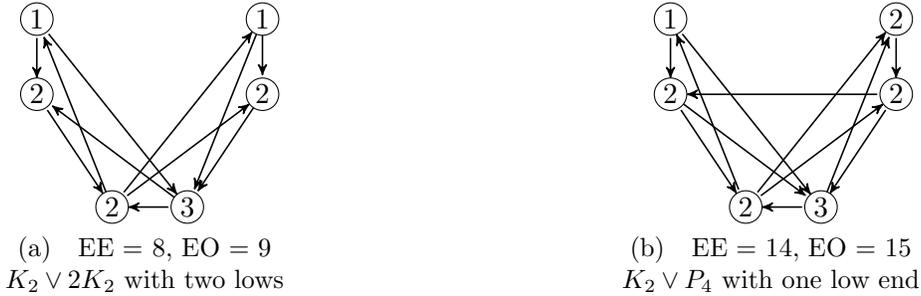

Recall our goal in this section: to write $G$ as a composition of linear
interval strips, where each strip is complete or complete less an edge (since
this implies that $G$ is very nearly a line graph).  Our main tool in this
endeavor is Lemma~\ref{Irreducible2Join}, which we will apply to each interval
2-join in the representation.  To this end, we would like that every interval
2-join is canonical and irreducible.  Of course, trivial 2-joins are fine also,
since their strips must be complete.  So in Lemma~\ref{TrivialOrCanonical}, we
show that every interval 2-join is either trivial or canonical.  Finally, in
the proof of Lemma~\ref{QuasiLineContainedInLine} we choose a composition
representation with the maximum number of strips; thus, every canonical
interval 2-join is irreducible.

\begin{lem}\label{TrivialOrCanonical}
Let $G$ be a BK-free graph with $\Delta(G) \ge 9$ and $\omega(G)<\Delta(G)$.
Each interval $2$-join in $G$ is either trivial $(A_1=A_2)$ or canonical
$(A_1\cap A_2=\emptyset)$.
\end{lem}  
\begin{proof}
Let $(H, A_1, A_2, B_1, B_2)$ be an interval $2$-join in $G$. Suppose that
$H$ is nontrivial ($A_1 \ne A_2$) and let $C \DefinedAs A_1 \cap A_2$. 
Now $(H \setminus C, A_1 \setminus C, A_2 \setminus C, C \cup B_1, C \cup B_2)$
is a canonical interval $2$-join.  We reduce this $2$-join until we get an
irreducible canonical interval $2$-join $(H', A_1', A_2', B_1', B_2')$ with $H'
\unlhd H \setminus C$.  Since $C$ is joined to $H\setminus C$, it is also
joined to $H'$.  Hence $C \subseteq B_1' \cap B_2'$. 
Now Lemma \ref{Irreducible2Join} implies $B_1' \cap B_2' = \emptyset$, 
so $A_1 \cap A_2 = C = \emptyset$. Thus, $H$ is canonical.
\end{proof}

\begin{lem}
\label{QuasiLineContainedInLine}
If $G$ is a quasi-line BK-free graph with $\Delta(G) \ge 9$ and
$\omega(G)<\Delta(G)$, then there is a line graph $G'$ with $G \subseteq
G'$ such that $\Delta(G') = \Delta(G)$ and $\omega(G') < \Delta(G')$.
\end{lem}

\begin{proof}
By Lemma \ref{NotCircularIntervalIfBKCritical}, $G$ is not a circular interval
graph.  By Lemma \ref{NoNonLinear}, $G$ has no non-linear homogeneous pair
of cliques.  So, by Theorem \ref{QuasilineStructure}, $G$ is a composition of
linear interval strips.  Choose such a composition representation of $G$ using
the maximum number of strips.  By Lemma~\ref{TrivialOrCanonical}, every interval
2-join is trivial or canonical.  
%Also, by the maximality of the representation, every canonical 2-join is irreducible.

Let $(H, A_1, A_2)$ be a strip in the composition.  
Let $B_1 \DefinedAs N_{G\setminus H}(A_1)$ and $B_2 \DefinedAs N_{G\setminus
H}(A_2)$.  Now $(H, A_1, A_2, B_1, B_2)$ is an interval $2$-join.  
If $A_1=A_2$, then $H$ is complete.  So suppose $A_1 \ne A_2$.  
Now $H$ is canonical, by Lemma \ref{TrivialOrCanonical}. If $H$ is reducible,
then by symmetry we can assume that $N_H(A_1) \setminus A_1 = N_H(v_1) \setminus
A_1$. But now replacing the strip $(H, A_1, A_2)$ with the two strips
$(G[A_1], A_1, A_1)$ and $(H\setminus A_1, N_H(A_1)\setminus A_1, A_2)$ gives a
composition representation of $G$ using more strips, which is a contradiction.  
Hence $H$ is irreducible.  Now
by Lemma \ref{Irreducible2Join}, $H$ is complete or $K_t - xy$ where $x$ and
$y$ are low in $G$ and $t \le 6$.
Thus, $G$ is a composition of strips, each of which is either complete or $K_t
- xy$, where $x$ and $y$ are low in $G$ and $t \le 6$.  

%Since $G$ is BK-free, it %has no induced copy of Figure
%\ref{fig:NoIncidentNonCompleteStrips}, which is %$f$-AT.  %Hence, no pair of
%incomplete strips are incident in the composition of $G$.  
Note that each vertex can play the role of $x$ or $y$ in at most one incomplete
strip.  So, we can add a matching containing $xy$ for each strip of
the form $K_t - xy$ without increasing the maximum degree. Let $G'$ be the
resulting graph.  Now we show that adding this matching does not create a
$K_{\Delta(G)}$.  For each strip of the form $K_t - xy$, exactly one
maximal clique in $G'$ contains both $x$ and $y$ (since $B_1\cap
B_2=\emptyset$) and this clique has at most 6 vertices.  Hence $\omega(G') \le
\max\{6,\omega(G)\}<\Delta(G)$, as desired.
\end{proof}

Lemma~\ref{QuasiLineContainedInLine}
completes our proof that if the list-coloring version of
the Borodin--Kostochka conjecture (or its paintability analogue) is true for
line graphs, then it is true for quasi-line graphs.  If $G$ is a
quasi-line counterexample, then the $G'$ guaranteed by the
lemma is a line-graph counterexample, since $\chil(G')\ge\chil(G)$; 
similarly, $\chi_{OL}(G')\ge \chi_{OL}(G)$.

\section{Line Graphs}
\label{line-graphs}
In this section we consider line graphs.  
The general idea is to show that if $G$ is a line graph of $H$, then some
subgraph of $G$ is a line graph of a bipartite graph $B$ such that
each edge of $B$ has many of its adjacent edges also in $B$.  We then use a
result of Borodin, Kostochka, and Woodall~\cite{BKW} to show that the line
graph of $B$ is $f$-AT or $f$-KP.  This is the first place in our proof that
relies heavily on subgraphs being kernel-perfect.  In particular, the key result
from~\cite{BKW}, shown below as Theorem~\ref{BKWmain}, has no analogue for
Alon--Tarsi orientations.  (One example of this is that the line graph of
$K_{3,3}$ has no Alon--Tarsi orientation with maximum outdegree at most 2.)
Our main result in this section is the following theorem.  

\begin{thm}
\label{mainLineGraphs}
If $G$ is a BK-free line graph, then $\Delta(G)<69$.
Thus, if $G$ is a line graph with $\Delta(G)\ge 69$, 
then $\chip(G)\le \max\{\omega(G),\Delta(G)-1\}$.  
\end{thm}
\noindent
In other words, we prove the Borodin--Kostochka conjecture, strengthened to
online list-coloring, for the class of line graphs with maximum degree at least
69.  (When combined with the previous section, this proves the same result for
the larger class of quasi-line graphs with maximum degree at least 69.)
Our proof relies mainly on the kernel method.  This technique came to
prominence when Galvin~\cite{Galvin} used it to prove the List Coloring
Conjecture for line graphs of bipartite graphs.  More precisely, he showed
that if $G$ is the line graph of a bipartite graph $H$, then $G$ is
$\Delta(H)$-edge-choosable.  A few years later Borodin, Kostochka, and
Woodall~\cite{BKW} sharpened Galvin's result.  They proved the following. 
(They only stated the result for list-coloring, but the same proof gives the
result for kernel-perfection.)
\begin{thm}[\cite{BKW}]
\label{BKWmain}
If $G$ is the line graph of a bipartite graph
$B$, then $G$ is $f$-KP, where $f(uv)=\max\{d_B(u),d_B(v)\}$ for every
edge $uv$ in $B$.  Thus, $G$ is $f$-paintable.
\end{thm}

This strengthening allowed for a surprisingly wide range of applications.  
One beautiful consequence of Theorem~\ref{BKWmain} is that
for every constant $k$ there exists a constant $\Delta_k$ such
that if $\mad(G)<k$ and $\Delta(G)\ge \Delta_k$, then $\chil'(G)=\chi'(G)$.
The main idea of the proof is to find as a subgraph of $G$ a certain type of
bipartite graph $B$ such that any coloring of $E(G)\setminus E(B)$ can be
extended to $E(B)$ by Theorem~\ref{BKWmain}.
Recently, Woodall~\cite{Woodall} gave a simpler proof of this result.  In that
paper he made explicit that it suffices to let $\Delta_k=\frac{k^2}2$.  Since all of
these proofs use the kernel method,
they extend directly to online list-coloring, as observed by
Schauz~\cite{Schauz-Paint-Correct}.  

Galvin's proof is well-known and it has
been widely reproduced (for example, in~\cite{PFTB} and~\cite{Diestel}).  The
proofs for the extensions by Borodin, Kostochka, and Woodall~\cite{BKW} and
Schauz~\cite{Schauz-Paint-Correct} are similar, so we do not reproduce
them here.  However, the result for bounded maximum average degree is much less
well-known.  (Further, we need one extra wrinkle, since the proofs in
\cite{BKW} and \cite{Woodall} give an upper bound on $\Delta(H)$.  We must
translate this to an upper bound on $\Delta(G)$, but this final step is
relatively easy.)  We particularly like Woodall's presentation, so
we follow that below, in Theorem~\ref{helperLineGraphs}.  

The proof of our main result in this
section has a simple outline.  Let $G$ be the line graph of some graph $H$.
In Lemmas~\ref{mu-at-most-3-AT}--\ref{6DegenerateHelper}, we show that if $G$ is BK-free,
then $H$ is 6-degenerate.  In particular, $\mad(H)<12$.  Next, in
Lemma~\ref{BKWalternater} and Theorem~\ref{helperLineGraphs}, we apply 
%the result of Borodin, Kostochka, and Woodall 
Theorem~\ref{BKWmain} to show that if $G$ is BK-free and
has $\mad(H)<12$, then $\Delta(G)\le 68$.  This completes the proof for line
graphs.  

Now we recall how this section fits into the larger context of the paper.
%In the first section, we showed that if there exists a BK-free claw-free graph
%$G$, then there also exists a BK-free quasi-line graph $G'$.
%\note{This is WRONG! Fix it.}
In the previous section, % of the paper, 
we showed that if there exists a BK-free quasi-line graph $G$ with $\Delta(G)\ge
9$, $\omega(G)<\Delta(G)$, and $\chi(G) > \max\{\omega(G),\Delta(G)-1\}$, then there
exists such a $G$ that is a line graph.  In fact, the proof constructs the line
graph with the same maximum degree as the original.  Thus, our result that
$\chip(G)\le\max\{\omega(G),\Delta(G)-1\}$ for every line graph $G$ with
$\Delta(G)\ge 69$ immediately extends to prove the same bound for every
quasi-line graph $G$ with $\Delta(G)\ge 69$.  Combining this result with our
previous work (Theorem~5.6\ in~\cite{BK-claw-free}), we get that $\chil(G)\le
\max\{\omega(G),\Delta(G)-1\}$ for every claw-free graph $G$ with $\Delta(G)\ge
69$.

In the rest of this section, we prove Theorem~\ref{mainLineGraphs}.
We begin with two lemmas showing that certain graphs are $d_1$-AT or $d_1$-KP.
The hypothesis $\omega(G)<\Delta(G)$ arises naturally
from our interest in the Borodin--Kostochka Conjecture.
When $G$ is a line graph of $H$, the edges incident to any
common endpoint in $H$ form a clique in $G$, so $\Delta(H) \le
\omega(G)<\Delta(G)$.
%Thus, $\omega(G)<\Delta(G)$ implies $\Delta(H)<\Delta(G)$.

\begin{lem}
\label{mu-at-most-3-AT}
If $G$ is BK-free with $\omega(G)<\Delta(G)$ and $G$ is the line graph of some
graph $H$, then $\mu(H)\le 3$.  Further, no edge of multiplicity 3 in $H$
appears on a triangle.
\end{lem}
\begin{proof}
Suppose, to the contrary, that $H$ has some edge $e$ of multiplicity at least
4; let $v\in V(G)$ be a vertex corresponding to $e$.  First suppose that
$d_G(v)=\Delta(G)$.  Now $G[\{v\}\cup N(v)]=\join{K_4}{B}$, for some graph $B$,
since $e$ has multiplicity at least 4.  Since $\omega(G)<\Delta(G)$, we get that
$\omega(B)\le \card{B}-2$.  Since $G$ is a line graph, $B$ has independence
number 2, so $B$ contains two disjoint pairs of non-adjacent vertices.  Thus,
$\join{K_4}{B}$ is $d_1$-AT, as shown in
Figure~\ref{ATpics1}\subref{K4v2E2c}--\subref{K2vC4}.
Now suppose instead that $d_G(v)=\Delta(G)-1$.  The argument is essentially the
same; however, now we only get that $\omega(B)\le \card{B}-1$, so $B$ is
incomplete.  Now $\join{K_4}B$ is $f$-AT, as shown in
Figure~\ref{ATpics1}\subref{K4vE2}, since $v$ is low.  This completes the proof
of the first statement.

Now we prove the second statement.  Suppose, to the contrary, that $H$ has an
edge $e$ of multiplicity 3 on a triangle.  Let $x_1, x_2, x_3$ be the vertices
of the triangle, with $x_1$ and $x_2$ the endpoints of $e$.  Let $v_1, v_2, v_3$
be the vertices corresponding to edges with endpoints $x_1$ and $x_2$.  
Let $v_4$ and $v_5$ be vertices corresponding to edges $x_1x_3$ and $x_2x_3$.  
Similar
to above, $G[\{v\}\cup N(v)]=\join{K_3}{B}$, for some graph $B$.  
If $d_G(v)=\Delta(G)-1$, then $\omega(G)\le \card{B}-1$.  So some edge of $H$
incident to $x_1$ or $x_2$ has an endpoint outside of $\{x_1,x_2,x_3\}$.  By
symmetry, say it is incident $x_1$; let $v_6$ be the corresponding vertex of
$G$.  
Now $v_1,\ldots,v_6$ induce in $G$ a subgraph that is $f$-AT, as
shown in Figure~\ref{ATpics1}\subref{K4vE2}, since $v_1$ is low.

Assume instead that $d_G(v)=\Delta(G)$, so $\omega(G)\le \card{B}-2$.
Recall that $\join{K_3}{P_4}$ is $d_1$-AT, as shown in
Figure~\ref{ATpics1}\subref{K3vP4}.  Suppose that an edge incident to $x_1$ has
an endpoint outside $\{x_1,x_2,x_3\}$ and also that an edge incident to $x_2$
has an endpoint outside $\{x_1,x_2,x_3\}$.  If these endpoints are distinct,
then $G$ has a copy of $\join{K_3}{P_4}$, which is $d_1$-AT.  If these endpoints
are identical, then $G$ has a copy of $\join{K_2}{C_4}$, which is $d_1$-AT, as
shown in Figure~\ref{ATpics1}\subref{K2vC4}.  So we conclude that either $x_1$
or $x_2$ has no incident edges with endpoints outside of $\{x_1,x_2,x_3\}$; by
symmetry, assume it is $x_2$.  
%Since $\omega(G)<\Delta(G)$, we know that $\mu(x_2x_3)\ge 2$.
Now we can view $\join{K_3}B$ as $\join{K_4}{(B-v_4)}$, since $v_4$ dominates $B$.
Since $\omega(B)\le |B|-2$, we conclude that $B-v_4$ contains two disjoint pairs of
nonadjacent vertices.  Thus, $\join{K_4}{(B-v_4)}$ is $d_1$-AT, as shown in
Figure~\ref{ATpics1}\subref{K4v2E2c}.
\end{proof}

Before proving our next lemma, we need a bit more information about kernel-perfect
orientations.
We can easily show that if $D$ is a kernel-perfect digraph, then every clique
of $D$ is oriented transitively (possibly with some bidirected edges);
otherwise $D$ would have some cyclically oriented 3-cycle, which has no kernel. 
Further, every directed odd cycle must have a chord.
In general, these condition are not sufficient to imply that $D$ is
kernel-perfect.  However Borodin, Kostochka, and Woodall~\cite{BKW-kp}
showed that if the underlying undirected graph $G$ of $D$ is a line graph, then
these conditions are indeed sufficient.

\begin{thm}[\cite{BKW-kp}]
\label{KP-iff-line}
Let $H$ be a line graph of a multigraph.  An orientation $D$ of $H$ is
kernel-perfect if and only if every clique of $H$ is transitively oriented
(possibly with some bidirected edges) and every directed odd cycle of $D$ has a
chord (also possibly bidirected).
\end{thm}

Now we use Theorem~\ref{KP-iff-line} to prove that three particular
line graphs have $f$-KP orientations, where $f(v)=d(v)$ for a few specified
vertices $v$ and $f(v)=d(v)-1$ otherwise.

\begin{lem}
\label{mu-3-KP}
The line graphs of the subgraphs shown in Figure~\ref{fig:mu-3-KP} are $f$-KP,
where $f(v)=d(v)$ for vertices corresponding to the six bold edges in (b) and $f(v)=d(v)-1$
otherwise.
%Let $H$ be a multigraph with $v_1,v_2,v_3\in V(H)$, $d(v_1)=7, d(v_2)=d(v_3)=5$,
%and $\mu(v_1v_2)=\mu(v_1v_3)=3$. 
%If $G$ is the line graph of $H$, then
%$G$ is $d_1$-AT.  Similarly, form $G'$ from $G$ by deleting its incident edge
%leading to a vertex other than $v_2$ and $v_3$.  Now $G'$ is $f$-AT, where
%$f(v)=d(v)$ for each vertex $v$ arising from an edge incident to $v_1$ in $H$
%and $f(v)=d(v)-1$ otherwise. 
\end{lem}

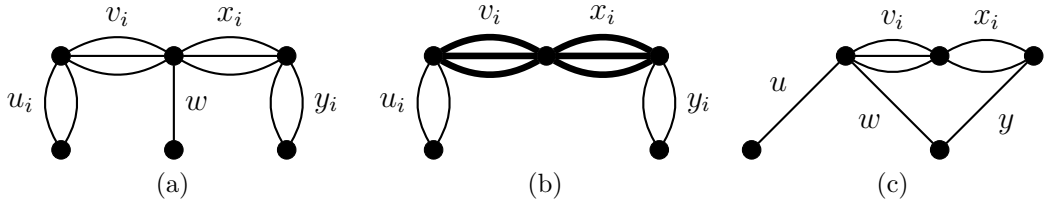
\begin{figure}[bht]
\centering
\subfloat[]{
\begin{tikzpicture}[scale = 5]
\tikzstyle{VertexStyle} = []
\tikzstyle{EdgeStyle} = []
\tikzstyle{labeledStyle}=[shape = circle, minimum size = 7pt, inner sep = 1.2pt, draw]
\tikzstyle{unlabeledStyle}=[shape = circle, minimum size = 7pt, inner sep = 1.2pt, draw, fill]
\Vertex[style = unlabeledStyle, x = 0.600, y = 0.750, L = \tiny {}]{v0}
\Vertex[style = unlabeledStyle, x = 0.900, y = 0.750, L = \tiny {}]{v1}
\Vertex[style = unlabeledStyle, x = 1.200, y = 0.750, L = \tiny {}]{v2}
\Vertex[style = unlabeledStyle, x = 0.900, y = 0.500, L = \tiny {}]{v3}
\Vertex[style = unlabeledStyle, x = 0.600, y = 0.500, L = \tiny {}]{v4}
\Vertex[style = unlabeledStyle, x = 1.200, y = 0.500, L = \tiny {}]{v5}

\Edge[labelstyle={auto=right, fill=none}](v1)(v0)
\Edge[style={bend left},  labelstyle={auto=right, fill=none}](v1)(v0)
\Edge[style={bend right}, label = {$v_i$},labelstyle={auto=right, fill=none}](v1)(v0)

\Edge[labelstyle={auto=right, fill=none}](v2)(v1)
\Edge[style={bend left},  labelstyle={auto=right, fill=none}](v2)(v1)
\Edge[style={bend right}, label = {$x_i$},labelstyle={auto=right, fill=none}](v2)(v1)

\Edge[style={bend left}, labelstyle={auto=right, fill=none}](v5)(v2)
\Edge[style={bend right}, label = {$y_i$}, labelstyle={auto=right, fill=none}](v5)(v2)

\Edge[style={bend left},  labelstyle={auto=right, fill=none}](v0)(v4)
\Edge[style={bend right}, label = {$u_i$}, labelstyle={auto=right, fill=none}](v0)(v4)

\Edge[label = {$w$}, labelstyle={auto=right, fill=none}](v3)(v1)
\end{tikzpicture}
}
\subfloat[]{
\begin{tikzpicture}[scale = 5]
\tikzstyle{VertexStyle} = []
\tikzstyle{EdgeStyle} = []
\tikzstyle{labeledStyle}=[shape = circle, minimum size = 7pt, inner sep = 1.2pt, draw]
\tikzstyle{unlabeledStyle}=[shape = circle, minimum size = 7pt, inner sep = 1.2pt, draw, fill]
\Vertex[style = unlabeledStyle, x = 0.600, y = 0.750, L = \tiny {}]{v0}
\Vertex[style = unlabeledStyle, x = 0.900, y = 0.750, L = \tiny {}]{v1}
\Vertex[style = unlabeledStyle, x = 1.200, y = 0.750, L = \tiny {}]{v2}
\Vertex[style = unlabeledStyle, x = 0.600, y = 0.500, L = \tiny {}]{v4}
\Vertex[style = unlabeledStyle, x = 1.200, y = 0.500, L = \tiny {}]{v5}

\Edge[style={color=black, line width=2.5pt}, labelstyle={auto=right, fill=none}](v1)(v0)
\Edge[style={bend left, color=black, line width=2.5pt},  labelstyle={auto=right, fill=none}](v1)(v0)
\Edge[style={bend right, color=black, line width=2.5pt}, label = {$v_i$},labelstyle={auto=right, fill=none}](v1)(v0)

\Edge[style={color=black, line width=2.5pt}, labelstyle={auto=right, fill=none}](v2)(v1)
\Edge[style={bend left, color=black, line width=2.5pt},  labelstyle={auto=right, fill=none}](v2)(v1)
\Edge[style={bend right, color=black, line width=2.5pt}, label = {$x_i$},labelstyle={auto=right, fill=none}](v2)(v1)

\Edge[style={bend left}, labelstyle={auto=right, fill=none}](v5)(v2)
\Edge[style={bend right}, label = {$y_i$}, labelstyle={auto=right, fill=none}](v5)(v2)

\Edge[style={bend left},  labelstyle={auto=right, fill=none}](v0)(v4)
\Edge[style={bend right}, label = {$u_i$}, labelstyle={auto=right, fill=none}](v0)(v4)
\end{tikzpicture}
}
\subfloat[]{
\begin{tikzpicture}[scale = 5]
\tikzstyle{VertexStyle} = []
\tikzstyle{EdgeStyle} = []
\tikzstyle{labeledStyle}=[shape = circle, minimum size = 7pt, inner sep = 1.2pt, draw]
\tikzstyle{unlabeledStyle}=[shape = circle, minimum size = 7pt, inner sep = 1.2pt, draw, fill]
\Vertex[style = unlabeledStyle, x = 0.650, y = 0.750, L = \tiny {}]{v0}
\Vertex[style = unlabeledStyle, x = 0.900, y = 0.750, L = \tiny {}]{v1}
\Vertex[style = unlabeledStyle, x = 1.150, y = 0.750, L = \tiny {}]{v2}
\Vertex[style = unlabeledStyle, x = 0.400, y = 0.500, L = \tiny {}]{v3}
\Vertex[style = unlabeledStyle, x = 0.900, y = 0.500, L = \tiny {}]{v4}
\Edge[label = {$u$}, labelstyle={auto=right, fill=none}](v0)(v3)

\Edge[labelstyle={auto=right, fill=none}](v1)(v0)
\Edge[style={bend left}, labelstyle={auto=right, fill=none}](v1)(v0)
\Edge[style={bend right}, label = {$v_i$}, labelstyle={auto=right, fill=none}](v1)(v0)

\Edge[style={bend left}, labelstyle={auto=right, fill=none}](v2)(v1)
\Edge[style={bend right}, label = {$x_i$}, labelstyle={auto=right, fill=none}](v2)(v1)

\Edge[label = {$y$}, labelstyle={auto=right, fill=none}](v4)(v2)
\Edge[label = {$w$}, labelstyle={auto=right, fill=none}](v0)(v4)
\end{tikzpicture}
}
\caption{The three cases of Lemma~\ref{mu-3-KP}.\label{fig:mu-3-KP}}
\end{figure}

%\begin{figure}[!htb]
%\includegraphics[scale=.4]{E.png}
%\includegraphics[scale=.4]{P_5-lows.png}
%\includegraphics[scale=.4]{C_4+pendant.png}
%\caption{The three cases of Lemma~\ref{mu-3-KP}.
%\label{fig:mu-3-KP}}
%\end{figure}

\begin{proof}
In each case, let $H$ denote the graph shown and let $G$ denote its line graph.
The orientation for the second line graph comes from the orientation for the
first, simply by deleting a vertex. Since the lists sizes don't go down, the
reducibility of the first line graph implies the reducibility of the second.  
Our orientations $D$ will actually orient some edges in both directions.
This is fine, as long as $d^+_D(v)\le d_G(v)-2$.

We begin with Figure~\ref{fig:mu-3-KP}(a).  We refer to the edges of the subgraph
and the vertices of its line graph interchangably.  From left to write, label
the edges as $u_1$, $u_2$, $v_1$, $v_2$, $v_3$, $w$, $x_1$, $x_2$, $x_3$, $y_1$,
$y_2$; if vertices differ only in their subscript, then they correspond to
parallel edges.

To form $D$, take all of the directed edges implied by transitivity in the three linear
orders $v_1\to v_2\to u_1\to u_2\to v_3$; $x_1\to x_2\to y_1\to y_2\to x_3$;
$v_3\to x_3\to w\to v_1\to v_2\to x_1\to x_2$ (one order for each maximal clique in
$G$).  Theorem~\ref{KP-iff-line} immediately implies that $D$ is
kernel-perfect, since $G$ is chordal.  All that remains is to verify that the
outdegrees are small enough.  In the table below we give
the degree of each vertex in $G$ and its outdegree in $D$.

\begin{figure}[hbt]
$$
\begin{array}{c|ccccccccccc}
& u_1 & u_2 & v_1 & v_2 & v_3 & w & x_1 & x_2 & x_3 & y_1 & y_2\\
\hline
d_G   & 4 & 4 & 8 & 8 & 8 & 6 & 8 & 8 & 8 & 4 & 4\\
d^+_D & 2 & 1 & 6 & 5 & 6 & 4 & 4 & 3 & 5 & 2 & 1
\end{array}
$$
\caption{Degrees in $G$ and outdegrees in $D$, for Figure~\ref{fig:mu-3-KP}(a).}
\end{figure}

%To complete the proof, we must show that $D$ is kernel-perfect.  Let $C$ be an
%arbitrary subset of $V(D)$.  We will show that $C$ has a kernel $S_C$.
%Typically, we try to partition $C$ into cliques, such that the final vertices of
%the cliques (under their transitive orderings) are independent.
%Whenever we choose some $u_i$ or $y_j$ to be in $S_C$, we always take $i$ and
%$j$ to be as large as possible.
%
%Suppose there exists $u_i,y_j\in C$.  If $w\in C$, then let $S_C=\{u_i,y_j,w\}$.
%So assume $w\notin C$.  Now we take as $S_C$ the first of the three following
%pairs that appears in $C$: $\{u_i,x_3\}$, $\{y_i, v_3\}$, and $\{u_i,v_j\}$.
%Thus, we assume that either $C$ contains no $u_i$ or $C$ contains no $y_j$.
%If $C$ contains neither $u_i$ nor $y_j$, then let $S_C$ be the last vertex of
%$C$ in the order $v_3, x_3, w, v_1, v_2$.  So $C$ must contain $u_i$ or $y_j$.
%Consider the former case. 
%If $C$ has a vertex among $x_3, w, x_1, x_2$, then let $S_C$ be the final vertex
%in that order, along with $u_i$.  Otherwise, let $S_C$ be the final vertex in
%the order $u_1, u_2, v_3$.  
%%let $S_C$ be $u_i$ and the last vertex of $C$ in the order
%%$v_3, x_3, w, v_1, v_2$.  
%%
%Consider the latter case. 
%If $C$ contains a vertex in $v_3, w, v_1, v_2$, then let $S_C$ be the final
%vertex in that order, along with $y_j$.  Otherwise, let $S_C$ be the final
%vertex in the order $y_1, y_2, x_3$.
\noindent
This completes the proof for
Figure~\ref{fig:mu-3-KP}(a), and also for Figure~\ref{fig:mu-3-KP}(b).

Now consider Figure~\ref{fig:mu-3-KP}(c).  
From left to write, label the edges as $u, v_1, v_2, v_3, w, x_1, x_2, y$.
To form $D$, take all directed edges implied by the four linear orders $v_3,w,
u, v_1, v_2$; $v_1, v_2, x_1, x_2, v_3$; $x_1, x_2, y$; $y,w$ (one order for
each maximal clique in $G$).  Again, Theorem~\ref{KP-iff-line} immediately
implies that $D$ is kernel-perfect (now $G$ is no longer chordal, but every
chordless cycle is even, which is sufficient).  Again, we need only verify that
the outdegrees are small enough.  In the table below we give the degree of each
vertex in $G$ and its outdegree in $D$.

\begin{figure}[hbt]
$$
\begin{array}{c|cccccccc}
& u & v_1 & v_2 & v_3 & w & x_1 & x_2 & y\\
\hline
d_G   & 4 & 6 & 6 & 6 & 5 & 5 & 5 & 3\\
d^+_D & 2 & 4 & 3 & 4 & 3 & 3 & 2 & 1
\end{array}
$$
\caption{Degrees in $G$ and outdegrees in $D$, for Figure~\ref{fig:mu-3-KP}(c).}
\end{figure}
\noindent
This completes the proof of the lemma.
%
%\begin{figure}[hbt]
%\centering
%\includegraphics[scale=.4]{C_4+pendant.png}
%\caption{Repeated from above for convenient viewing.}
%\end{figure}
%
%To complete the proof, we must show that $D$ is kernel-perfect.  Let $C$ be an
%arbitrary subset of $V(D)$.  We will show that $C$ has a kernel $S_C$.
%
%First suppose that $y\in C$.  Form $S_C$ as follows. Take the last vertex of $C$ 
%in the order $v_3, w, u, v_1, v_2$ (if it exists).  When this last vertex is
%not $w$, add $y$ to $S_c$.  When it is $w$, if some $x_i\in C$, then add $x_i$
%to $C$.  So assume that $y\notin C$.  If $x_1, x_2\notin C$, then let $S_C$ be
%the last vertex of $C$ in the order $v_3, w, u, v_1, v_2$.  So assume that
%$x_i\in C$.  If $u\in C$ or $w\in C$, then let $S_C$ be $x_i$ and the last vertex
%of $C$ in the order $w, u$.  Now $C\subseteq \{v_1, v_2, x_1, x_2, v_3\}$, so
%let $S_C$ be the last vertex of $C$ in the order $v_1, v_2, x_1, x_2, v_3$. 
%%
\end{proof}

If one or both pairs of parallel edges incident to leaves in
Figure~\ref{fig:mu-3-KP}(a) ended instead at distinct leaves, then the resulting line
graph would be unchanged; so it is again $f$-KP.  
In proving our next lemma, we use this observation implicitly.

\begin{lem}
\label{6DegenerateHelper}
Let $G$ be the line graph of some graph $H$.
If $\delta(H) \ge 7$ and $\mu(H) \le 3$, then 
%$L(H)$ has a $d_1$-paintable induced subgraph (or one with lows).
$G$ is not BK-free.
Thus, if $G$ is BK-free and $\omega(G)<\Delta(G)$, then $H$ is 6-degenerate.
\label{6Degenerate}
\end{lem}
\begin{proof}
We begin by proving the first statement.
%Put $G = L(H)$. 
Choose a partition $\{A, B\}$ of $V(H)$ to 
\begin{enumerate}
\item[(1)] maximize $||A, B||$; and subject to that to 
\item[(2)] minimize $\sum_{xy \in E(A, B)} \mu(xy)^2$
\end{enumerate}
\noindent
Here (2) is just giving preference to things like 3 single edges over one triple edge.  

Let $Q$ be the bipartite graph with parts $A$ and $B$ and edges $E(A, B)$.  Note
that $d_Q(x) \ge d_H(x) / 2$ for all $x \in V(Q)$ by (1); otherwise we could
move $x$ to the other part and increase $||A, B||$.  For each $x \in V(Q)$, let
$\mu(x)$ be the maximum multiplicity of an edge in $Q$ incident to $x$.
We apply Theorem~\ref{BKWmain} to show that the line graph of $Q$ is a
$d_1$-KP subgraph of $G$ (or else $G$ contains some subgraph from
Lemma~\ref{mu-3-KP} that is $d_1$-KP). 

The hypothesis for Theorem~\ref{BKWmain} requires that $\max\{d_Q(x),d_Q(y)\}\le
(d_Q(x)+d_Q(y)-2-(\mu(xy)-1))-1$ if edge $xy$ is high.  So it suffices to show
that $d_Q(x)\ge \mu(xy)+2$ and $d_Q(y)\ge \mu(xy)+2$.
Similarly, if $xy$ is a low edge, then we need 
$d_Q(x)\ge \mu(xy) + 1$ and $d_Q(y)\ge \mu(xy) + 1$.
Since $\mu(H)\le 3$, and $d_Q(z)\ge d_H(z)/2$ for all $z\in V(Q)$, it would
suffice to have $\delta(H) \ge 9$.  Thus, we may assume $\delta(H) \le 8$.
Since $\delta(Q) \ge 4$,  we may apply Theorem~\ref{BKWmain} unless there is a
vertex $x$ with $d_Q(x) = 4$ incident to a high edge $xy$ with $\mu(xy) = 3$.  
So suppose this is true.  We have two cases: $d_H(x)=8$ and $d_H(x)=7$.

\textbf{$d_H(x) = 8$}: By (2), $x$ must be incident to two multiplicity
3 edges and two multiplicity 1 edges; otherwise we could move $x$ to the other
part of the partition and contradict that the partition is extremal.  But now we
have a $d_1$-AT subgraph, by Lemma~\ref{mu-3-KP}.

\textbf{$d_H(x) = 7$}: We have $d_H(x) + d_H(y) - \mu(xy) - 1 = \Delta(G)$. 
Since $d_H(x)=7$, we get $d_H(y)=\Delta(G)-3$.
Pick $w \in N_H(y) - x$.  Now $\Delta(G)\ge d_H(y) + d_H(w) - \mu(yw) - 1$, 
so $\mu(yw) + 4\ge d_H(w) \ge 7$.
Thus $\mu(yw) \ge 3$.  But now $y$ is incident to two multiplicity 3 edges, so
we again have a $d_1$-AT subgraph, by Lemma~\ref{mu-3-KP}.
%I proved this in an email before, it is pretty easy.  The main thing we
%need is that the attached pictures are $d_1$-paintable (where green means low). 
%A similar idea to BKW proves these.  You gave reference colorings before that
%worked, but really the reference coloring business is a (obfuscating) red
%herring, all we need to do is pick a total order for the edges incident to each
%vertex and use the corresponding orientation.

This second statement of the lemma follows from the first one.  Note that
$\mu(H)\le 3$, by Lemma~\ref{mu-at-most-3-AT}.  If $H$ has a subgraph $H'$ with
$\delta(H')\ge 7$, then we apply the first part of the lemma to $H'$ and
conclude that $G$ is not BK-free.  Instead every subgraph $H'$ of $H$ must
have a vertex of degree at most 6.  Thus, by definition, $H$ is 6-degenerate.
\end{proof}

%Probably can be simpler.  Can get a lot of info in the $\delta(H) = 6$ case,
%but not finished (for instance, a multiplicity 3 edge incident to a 6 vertex
%forces a cycle blown up to triangles).

\begin{lem}
Let $G$ be the line graph of some graph $H$ such that $\Delta(H) < \Delta(G)$. 
If $G$ is BK-free, then $H$ has no bipartite subgraph $B$ such that for every
edge $xy \in E(B)$ the endpoint of smaller degree has all of its incident edges
in $H$ also appearing in $B$.  In other words, $H$ cannot contain a bipartite
subgraph $B$ such that for each $xy\in E(B)$ 
either
(a) $d_B(y) \ge d_B(x) = d_H(x)$; or
(b) $d_B(x) \ge d_B(y) = d_H(y)$.
\label{BKWalternater}
\end{lem}
\begin{proof}
Suppose, to the contrary, that we have graphs $G$, $H$, and $B$ as in the
lemma.  Let $D$ denote the line graph of $B$. We use
Theorem~\ref{BKWmain} to show that $D$ is 
$f$-AT, where $f(v):=d_D(v)-1+\Delta(G)-d_G(v)$ for all $v\in V(D)$.
This contradicts the fact that $G$ is BK-free.

Consider a vertex $v$ of $D$ and let $xy$ denote the corresponding edge in $B$.
To apply Theorem~\ref{BKWmain}, we must show that $f(v)$ is sufficiently
large; namely, we must show that $d_D(v)-1+\Delta(G)-d_G(v)\ge
\max\{d_B(x),d_B(y)\}$.  By symmetry, assume that $d_B(y) \ge d_B(x) =
d_H(x)$.  Recall that $d_D(v)=d_B(x)+d_B(y)-\mu(xy)-1$.
Now $f(v)=d_D(v)-1+\Delta(G)-d_G(v) = (d_B(x)+d_B(y)-\mu(xy)-1)-1 +
\Delta(G)-(d_H(x)+d_H(y)-\mu(xy)-1) =d_B(y)-1+\Delta(G)-d_H(y)\ge d_B(y)$,
since $\Delta(H)<\Delta(G)$.
\end{proof}

To simplify our presentation of the key lemma from \cite{BKW} and
\cite{Woodall}, we state it only for the specific case needed for our
application: $\mad(G)<12$.  However, the proof extends easily to the more
general case that $\mad(H)<C$, for some constant $C$.

\begin{lem}
\label{helperLineGraphs}
Suppose that $G$ is BK-free, $G$ is the line graph of a graph $H$, and $\mad(H)<12$. 
If $\omega(G)<\Delta(G)$, then $\Delta(G)\le 68$.  
%Let $G$ be a BK-free line graph with $G=L(H)$.  If $\mad(H)<12$, then 
%$\Delta(G)\le 67$.
%If $G$ is a BK-free quasi-line graph with $\Delta(G) \ge 68$, then
%$G$ contains $K_{\Delta(G)}$.
\end{lem}
\begin{proof}
%Suppose not. By the stuff in the document, $G$ must be a line graph, say
%$G = L(H)$.  

Suppose the lemma is false, and let $G$ be a counterexample.
%Choose $G$ such that $G$ is BK-free, 
%$G$ is the line graph of a graph $H$, 
%$\mad(H)<12$, $\omega(G)<\Delta(G)$, and $\Delta(G)\ge 69$.  
We use the discharging method to get a contradiction;  since we know
$\mad(H)<12$, we use discharging on the vertices of $H$.
Give each vertex $v\in V(H)$ initial charge $\ch(v)=d_H(v)$.  
%We follow Woodall's presentation, using iterated discharging; 
We have 10 successive rounds of discharging, rounds 2 through 11.  On round
$i$, each vertex of degree at most $i$ receives charge 1 from some high degree
neighbor.  Thus, for each $v$ with $d(v)\le 11$, vertex $v$ receives charge 1
on a total of $12-d(v)$ rounds.  Hence, each such vertex receives total
charge $12-d(v)$, and ends with final charge $d(v)+(12-d(v))=12$.  We also must
verify that no vertex gives away too much charge.  Thus, each vertex finishes
with charge at least 12, which contradicts our assumption that $\mad(H)<12$.
The details forthwith.

For each round of discharging, we use Lemma~\ref{BKWalternater} repeatedly.  
For each $i$ with $2 \le i \le 11$, let $V_i$ be the vertices of $H$ of degree
at most $i$.  Let $B_i$ be the bipartite subgraph of $H$ containing $V_i$ and
all edges incident to $V_i$.
Since $G$ is BK-free, $\delta(G)\ge \Delta(G)-1$.  Thus, each edge
$uv$ in $H$ has $d(u)+d(v)-2\ge \Delta(G)-1$, so $d(u)+d(v)\ge \Delta(G)+1$. 
Since $\Delta(H)\le \omega(G)<\Delta(G)$, we have $d(u)+d(v)\ge \Delta(H)+2$ for
each edge $uv$.  In particular, $\delta(H)\ge 2$.  This also implies that for
each $i$ with $2\le i\le 11$ the set $V_i$ is independent in $H$.  

For every $u \in V_i$, we have $d_{B_i}(u) = d_H(u)$, so
Lemma~\ref{BKWalternater} shows that there must exist $v \in V(B_i)
\setminus V_i$ with $d_{B_i}(v) < d_{B_i}(u) \le i$.  Thus, on round $i$, we
give charge 1 from $v$ to each of its neighbors in $B_i$; afterwards, 
we delete from $B_i$ vertex $v$ and all of its neighbors in $B_i$.
Now again applying Lemma~\ref{BKWalternater} gives another vertex $v'$ with
$d_{B_i - v - N(y)}(v') < i$. We can repeat this process until
$B_i \cap V_i$ is empty, at which time each $v \in V_i$ has received charge 1.
On round $i$, each $v \in V(B_i) \setminus V_i$ has lost charge at most $i-1$,
since it gave charge 1 to at most $i-1$ neighbors.

Recall that $d(u)+d(v)\ge \Delta(G)+1 \ge \Delta(H)+2$ for each edge $uv\in E(H)$.
Thus, on round 2, only $\Delta(H)$-vertices give
charge (and only $2$-vertices receive it). Analogously, on an arbitrary
round $i$, only vertices of degree at least $\Delta(H) + 2 - i$ give
charge.  So if a vertex gives charge only on rounds $i$ through 11, then it
gives away charge at most $(i-1)+i+\cdots+10$.  Since charge is first given on
round $\delta(H)$,
in general each vertex loses at most $55 - (1 + 2 + \cdots + (\delta(H) - 2))$.  
This maximum amount of charge can only be lost by a vertex of degree at least
$\Delta(H)-\delta(H)+2$.  So, if some vertex of $H$ finishes the discharging
rounds with insufficient charge, then
$(\Delta(H)-\delta(H)+2)-(55 - (1 + 2 + \cdots + (\delta(H) - 2)))\le 11$.
This simplifies to $\Delta(H)+\frac{(\delta(H)-2)(\delta(H)-3)}2\le 66$.  Thus,
if $\delta(H)\ge 5$, we have $\Delta(H)+\delta(H)-2\le 66$; if $\delta(H)=4$, then
$\Delta(H)\le 65$. Finally, if $2\le \delta(H)\le
3$, then we still have $\Delta(H)\le 66$.

Now we are almost done.  However, we must still translate our upper bound on $\Delta(H)$
into an upper bound on $\Delta(G)$.
Let $u$ be a minimum degree vertex in $H$ and $v$ a neighbor of $u$.  Then, in $G$
we have $d_G(uv) = d_H(u) + d_H(v) - 1 - \mu(uv) \le \delta(H) + \Delta(H) - 2$.
Since $G$ is BK-free, every vertex has degree at least $\Delta(G) - 1$.  So 
$\Delta(G) - 1 \le \delta(H) + \Delta(H) - 2$.  Now we apply the bounds from the
previous paragraph.
If $\delta(H)\ge 5$, then
$\Delta(G)\le \delta(H)+\Delta(H)-1\le 67$.  If, instead, $\delta(H)=4$, then
$\Delta(G)\le 65+4-1=68$.  Finally, if $\delta(H)\le 3$, then $\Delta(G)\le
66+3-1=68$.
\end{proof}

Now we combine Lemmas~\ref{mu-at-most-3-AT}--\ref{helperLineGraphs} to prove
Theorem~\ref{mainLineGraphs}.  For convenience, we restate it.

\begin{LineThm}
If $G$ is a BK-free line graph with $\omega(G)<\Delta(G)$, then $\Delta(G)<69$.
Thus, if $G$ is a line graph with $\Delta(G)\ge 69$, then
$\chip(G)\le \max\{\omega(G),\Delta(G)-1\}$.  
\end{LineThm}
\begin{proof}
Let $G$ be a BK-free graph such that $G$ is the line graph of some graph $H$.
First, suppose that $\omega(G)<\Delta(G)$.  Now $H$ is 6-degenerate, by
Lemma~\ref{6DegenerateHelper}, so $\mad(H)<12$;  
%~\ref{mu-at-most-3-AT} and
thus, the first statement follows from Lemma~\ref{helperLineGraphs}.
Now consider the second statement.
If $\omega(G)\ge\Delta(G)$, then the result holds by Brooks' Theorem
(more precisely, its generalization to Alon--Tarsi orientations, proved
in~\cite{Brooks-AT}).  So assume that $G$ is a minimal counterexample; now 
$\omega(G)<\Delta(G)$ and $\Delta(G)\ge 69$. 
The minimality of $G$ implies that $G$ is BK-free.  Now the first statement
implies that $\Delta(G)<69$, which is a contradiction.
\end{proof}

\begin{cor}
%If $G$ is a BK-free quasi-line graph with $\omega(G)<\Delta(G)$, then $\Delta(G)<69$.
If $G$ is a quasi-line graph with $\Delta(G)\ge 69$, then $\chip(G)\le
\max\{\omega(G),\Delta(G)-1\}$.  Further, if $G$ is a claw-free
graph with $\Delta(G)\ge 69$, then $\chil(G)\le \max\{\omega(G),\Delta(G)-1\}$.  
\end{cor}
\begin{proof}
The first statement follows from Lemma~\ref{QuasiLineContainedInLine}.
The second statement follows from a similar reduction from claw-free graphs to
quasi-line graphs for the list-coloring version of the Borodin-Kostochka
conjecture, which we proved in~\cite[Theorem 5.6]{BK-claw-free}.
\end{proof}

%By Lemma 1, $H$ is 6-degenerate and hence the average degree of $H$ is less than
%12. Give each vertex charge equal to its degree. Now we do Woodall's iterated
%discharging for 10 iterations, at that point the lowest degree vertices (the
%2-vertices) will all have charge 12.  We show that none of the vertices that
%gave charge go below 12, so this contradicts our average degree condition.

%\input{ATPics}

%\begin{lem}
%\label{NonLinearPicsLemma}
%The following configurations are $f$-AT with the given values of $f$: (a) a
%4-cycle with an extra vertex that is adjacent to three vertices on the cycle,
%where $f(v)=d(v)-1$ for one neighbor of the 2-vertex and $f(v)=d(v)$ otherwise;
%(b) $K_6-2P_3$, where $f(v)=d(v)-1$ for the two 3-vertices and $f(v)=d(v)$
%otherwise; (c) the rightmost configuration in Figure~\ref{NonLinear-fig}, with $f(v)=d(v)-1$ for all $v$.
%\end{lem}

%\clearpage
\bibliographystyle{amsplain}
\bibliography{GraphColoring}
\end{document}